\documentclass[12pt, reqno]{amsart}
\usepackage{amssymb, hyperref}
\usepackage{amsthm}
\usepackage{amsmath}
\usepackage{graphicx}
\usepackage[all,2cell]{xy}
\usepackage{verbatim}
\usepackage{tikz}
\usepackage{tikz-cd}
\usepackage{hyperref}
\usepackage{array}
\usepackage[capitalise]{cleveref}
\usepackage{epsfig,graphicx}
\usepackage{tikz}
\usepackage{tikz-cd}
\usepackage{float}
\usepackage{capt-of}
\tikzset{node distance=2cm, auto}
\usepackage{etoolbox}
\usetikzlibrary{graphs}
\usepackage{listings}
\usetikzlibrary{backgrounds}
\usetikzlibrary{decorations.markings}
\tikzstyle{vertex}=[circle, draw, inner sep=0pt, minimum size=6pt]
\usepackage{geometry}
\usepackage{mathrsfs}
\usepackage{caption}
\numberwithin{equation}{section}
\newtheorem*{theorem*}{Theorem}
\newtheorem*{corollary*}{\bf Corollary}
\newtheorem*{remark*}{\bf Remark}
\usepackage{lipsum}
\usetikzlibrary{matrix,arrows,decorations.pathmorphing}
\usepackage{lineno}
\newtheorem{theorem}{Theorem}[section]
\newtheorem{corollary}[theorem]{Corollary}

\newtheorem{example}[theorem]{Example}
\newtheorem{lemma}[theorem]{Lemma}

\newtheorem{proposition}[theorem]{Proposition}
\newtheorem{remark}[theorem]{Remark}

\newcommand{\eat}[1]{}

\begin{document}
	\author[A. Nayek]{Arpita Nayek}
	
	\address{SRM University-AP, Neerukonda, Guntur, Andhra Pradesh 522240, India}
	\email{arpita.n@srmap.edu.in}

	\author[A. J. Parameswaran]{A. J. Parameswaran}
	\address{Tata Inst. of Fundamental Research,
		Homi Bhabha Road, Colaba
		Mumbai 400005, India}
	\email{param@math.tifr.res.in}
	
	\author[P. Saha]{Pinakinath Saha}
	\address{Tata Inst. of Fundamental Research,
		Homi Bhabha Road, Colaba
		Mumbai 400005, India}
	\email{psaha@math.tifr.res.in}

	\title{On Automorphism group of a $G$-induced variety}
	\begin{abstract} 	
		Let $G$ be a connected semisimple algebraic group of adjoint type over the field $\mathbb{C}$ of complex numbers and $B$ be a Borel subgroup of $G.$ Let $F$ be an irreducible projective $B$-variety. Then consider the variety $E:=G\times^{B}F,$ which has a natural action of $G$; we call it $G$-induced variety or $(G,B)$-induced variety. In this article, we compute the connected component containing the identity automorphism of the group of all algebraic automorphisms of some particular $G$-induced varieties $E.$ 
	\end{abstract}
	
	\subjclass[2010]{14M15}
	
	\keywords{G-Schubert variety, Tangent sheaf, Automorphism group, Locally rigid}
	\maketitle
	\tableofcontents
	\section{Introduction}
	Let $X$ be a projective variety over complex numbers $\mathbb{C}.$ Let ${\rm Aut}^0(X)$ be the connected component, containing the identity automorphism of the group of all algebraic automorphisms of $X.$	Then ${\rm Aut}^0(X)$ has a structure of an algebraic group (see \cite[Theorem 3.7, p.17]{MO}). Further, the Lie algebra of this automorphism is isomorphic to the space of all tangent vector fields on $X,$ that is the space $H^0(X,\Theta_{X})$ of all global sections of the tangent sheaf $\Theta_{X}$ of $X$ (see \cite[Lemma 3.4, p.13]{MO}).

	Let $G$ be a connected semisimple algebraic group of adjoint type over $\mathbb{C}.$ Demazure \cite{Dem aut} studied automorphism group of a partial flag variety, i.e., a homogeneous variety of the form $G/P,$ where $P$ is a parabolic subgroup of $G.$ Further, Demazure proved that all the higher cohomology groups of the tangent bundle of a partial flag variety vanish. Bott proved this in the complex analytic setup in \cite[Theorem VII, p.242]{Bot}. As a particular case of his result, it follows that the connected component containing the identity automorphism of the group of all algebraic automorphisms of a full flag variety (i.e., a homogeneous variety of the form $G/B$, where $B$ is a Borel subgroup of $G$) is identified with $G.$
	
	By Kodaira-Spencer theory, the vanishing of the first cohomology group of the tangent bundle of a partial flag variety implies that partial flag varieties admit no local deformation of their complex structure. In other words, for any continuous family of complex varieties $X_{y}$ parameterized by a complex variety $Y,$ where $X_{y}$ is topologically isomorphic to $X$ for all $y,$ and $X_{0}$ is analytically isomorphic to $X,$ then $X_{y}$ is analytically isomorphic to $X$ in a neighborhood of $0\in Y.$
	
	Let $B$ be a Borel subgroup of $G.$ Let $F$ be a projective $B$-variety. Consider the variety
	\begin{equation*}
		E:=G\times^{B} F=G\times F/\sim,
	\end{equation*}
	where the action of $B$ on $G\times F$ is given by $b\cdot(g, f)=(gb^{-1}, bf)$ for all $g\in G, b\in B,$ $f\in F$ and $``\sim"$ denote the equivalence relation defined by the action.	The equivalence class of $(g, f)$ is denoted by $[g, f].$ Note that there is a natural action of $G$ on $E$ given by $g'\cdot[g,f]=[g'g,f],$ where $g'\in G, [g,f]\in E.$ Then $E$ is a projective variety together with a $G$-action on it; we call it $(G, B)$-induced variety. Throughout this article we use the terminology $G$-induced variety instead of $(G, B)$-induced variety for the sake of simplicity.

	In this article, we study the connected component containing the identity automorphism of the group of all algebraic automorphisms of some particular $G$-induced varieties.
	
	Let $V$ be a $B$-module. Let $\mathcal{L}(V)$ be the associated homogeneous vector bundle on $G/B$ corresponding to the $B$-module $V.$ We denote the cohomology modules $H^j(G/B, \mathcal{L}(V))$ $(j\ge 0)$  by $H^j(G/B, V)$ $(j\ge 0)$ for short. 
	
	Our main results of this article  are the following.
	
	\begin{theorem}[See Theorem \ref{thm 3.3}]\label{thm 1.1}
		Let $F$ be an irreducible projective $B$-variety. 
		Let $E=G \times^{B} F$ be the $G$-induced variety associated to $F$. Let $\Theta_{E}$ (respectively, $\Theta_{F}$) be the
		tangent sheaf of $E$ (respectively, of $F$).  Then we have
		\begin{itemize}
			\item [(i)] ${\rm Aut}^{0}(E)=G,$ if $H^{0}(G/B, H^{0}(F, \Theta_{F}))=0.$ 
			
			\item [(ii)] Assume that $H^j(F,\mathcal{O}_{F})$ vanish for all $j\ge 1,$ where $\mathcal{O}_{F}$ denotes the structure sheaf on $F.$ Then $H^{1}(E, \Theta_{E})=H^0(G/B, H^{1}(F, \Theta_{F})),$ if $H^j(G/B, H^0(F,\Theta_{F}))=0$ for $j = 1, 2.$   
	\end{itemize}
	\end{theorem}
	The hypotheses of Theorem \ref{thm 1.1}(ii) are  satisfied by a very special class of varieties namely the unirational varieties which also includes flag varieties, Schubert varieties, Bott-Samelson-Demazure-Hansen varieties (see \cite[Lemma 1, p.481]{Ser}). Under this assumption Theorem \ref{thm 1.1}(ii) allows us to compare the local deformation of $E$ and the local deformation of the fibre space $F$ relative to the base space $G/B$.

	Let $T$ be a maximal torus of $G$ and $R$ be the set of roots with respect to $T.$ Let $R^{+}\subset R$ be a set of positive roots. Let $B^{+}$ be the Borel subgroup of $G$ containing $T,$ corresponding to $R^{+}.$ Let $B$ be the Borel subgroup of $G$ opposite to $B^{+}$ determined by $T.$ Let $W=N_{G}(T)/T$ denote the Weyl group of $G$ with respect to $T$, where $N_{G}(T)$ denotes the normalizer of $T$ in $G$. For $w\in W,$ let $X(w):= \overline{BwB/B}$ denote the Schubert variety in $G/B$ corresponding to $w.$ 
	
	Consider the diagonal action of $G$ on $G/B\times G/B.$ Then there is a $G$-equivariant isomorphism $$\xi : G\times^{B} G/B \longrightarrow G/B\times G/B$$ given by $$[g , g'B]\mapsto (gB, gg'B),$$ where $g,g'\in G.$ 
	
	 For any $w\in W,$ $\xi(G\times^{B} X(w))$ is $G$-stable closed irreducible subset of $G/B\times G/B$. Moreover all closed irreducible $G$-stable subsets of $G/B\times G/B$ are precisely of the form $\big\{ \xi(G\times^{B}X(w)): w\in W\big\}$ (see \cite[Definition 2.2.6, p.69-70]{BK}).
	
	For $w\in W,$ let $\mathcal{X}(w):=\xi (G\times^{B}X(w)).$ Then $\mathcal{X}(w)$ is equipped with the structure of a closed subvariety of $G/B\times G/B,$ this $G$-induced variety is called $G$-Schubert variety associated to $w.$ Now onwards we omit $\xi$ and simply write $\mathcal{X}(w)$ to be $G\times^{B} X(w).$ Then
	we prove 
	
	\begin{proposition}[See Proposition \ref{prop 4.3}]\label{prop 1.3}
		Assume that $G$ is simply-laced. Let $w\in W$ be
		such that $w\neq w_{0},$ where $w_{0}$ denotes the longest element of $W.$ Let $\Theta_{\mathcal{X}(w)}$ $($respectively, $\Theta_{X(w)})$ be the tangent sheaf of $\mathcal{X}(w)$ $($respectively, of $X(w)).$ Then we have
		\begin{itemize}
			\item [(i)] ${\rm Aut}^0(\mathcal{X}(w))= G.$
			
			\item [(ii)] $H^1(\mathcal{X}(w),\Theta_{\mathcal{X}(w)})=H^0(G/B, H^1(X(w) ,\Theta_{X(w)})).$
		\end{itemize}
	\end{proposition}

	Thus if $G$ is simply-laced and $w\neq w_{0}\in W,$ then by Proposition \ref{prop 1.3}, we conclude that $\mathcal{X}(w)$ admits no local deformation whenever $X(w)$ does so.
	
	Let $w=s_{i_1}s_{i_2}\cdots s_{i_r}$ be a reduced expression and let $\underline{
		i}:=(i_1, \ldots, i_r).$ Let $Z(w,\underline{
		i})$ be the Bott-Samelson-Demazure-Hansen variety (natural desingularization of $X(w)$) associated to $(w, \underline{i}).$ It was first introduced by Bott and Samelson in a differential geometric and topological context (see \cite{BS}). Demazure \cite{Dem1} and Hansen \cite{Han} independently adapted the construction in algebro-geometric situation, which explains the reason
	for the name. For the sake of simplicity, we write BSDH-variety instead of Bott-Samelson-Demazure-Hansen variety.
	
	There is a natural left action of $B$ on $Z(w,\underline{i}).$ Let $\mathcal{Z}(w, \underline{i})= G\times^{B}Z(w, \underline{i}).$ Then the
	$G$-induced variety $\mathcal{Z}(w,\underline{i})$ is a smooth projective variety and it is a natural desingularization of $\mathcal{X}(w)$ (see \cite[Corollary 2.2.7, p.70]{BK}), we call it $G$-Bott-Samelson-Demazure-Hansen variety ($G$-BSDH-variety for short). Then we prove 
	
	\begin{proposition}[See Proposition \ref{prop 4.9}]\label{prop 1.5}
		Assume that $G$ is simply-laced and the rank of $G$ is at least two. Let $\Theta_{\mathcal{Z}(w,\underline{i})}$ be the
		tangent sheaf on $\mathcal{Z}(w, \underline{i}).$ Then we have 
		\begin{itemize}
			\item [(i)] ${\rm Aut}^0(\mathcal{Z}(w, \underline{i}))=G.$
			
			\item [(ii)] $H^{j}(\mathcal{Z}(w, \underline{i}), \Theta_{\mathcal{Z}(w,\underline{i})})=0$ for $j\ge 1.$
		\end{itemize}	
	\end{proposition}
	
	By Proposition \ref{prop 1.5}(ii), $H^2(\mathcal{Z}(w, \underline{i}), \Theta_{\mathcal{Z}(w,i)})=0.$ Hence by \cite[p.273]{Huy}, we conclude that $\mathcal{Z}(w, \underline{i})$ has unobstructed deformation for a simply-laced group $G.$ 
	
	Further, by Proposition \ref{prop 1.5}(ii), $H^1(\mathcal{Z}(w, \underline{
		i}), \Theta_{\mathcal{Z}(w,\underline{
			i})})= 0.$ Hence by \cite[Proposition 6.2.10, p.272]{Huy}, we conclude that $G$-BSDH-varieties are locally rigid for simply-laced group $G.$
	
	It should be mentioned here that if $G$ is not simply-laced, then $H^1(\mathcal{Z}(w, \underline{i}), \Theta_{\mathcal{Z}(w,\underline{i})})$
	might be non-zero (see Example \ref{ex4.12}).
	
	In the view of the above results the following questions are open:

	{\bf Open problems:} 
	\begin{itemize}
		\item[(1)] Assume that $G$ is not simply-laced. What is the connected component containing the identity automorphism of the group of all algebraic automorphisms of $\mathcal{X}(w)$ or $\mathcal{Z}(w, \underline{i})?$
		
		\vspace{.1cm}
		\item [(2)] What is the group of all algebraic automorphisms of $\mathcal{X}(w)$ or $\mathcal{Z}(w, \underline{i})?$
	\end{itemize}
	
	The organization of the paper is as follows. In Section 2, we set up notation and recall some preliminaries. In Section 3, we prove Theorem \ref{thm 1.1}. In Section 4, we prove Proposition \ref{prop 1.3} and Proposition \ref{prop 1.5}.
	
 {\it Acknowledgement.} We are grateful to the referee for his/her careful reading and numerous valuable remarks and comments. We are also thankful to Professor Michel Brion for his encouragement and useful comments through e-mail exchanges on a preliminary version of this paper.
 The first named author would like to thank Chennai Mathematical Institute for the postdoctoral fellowship and the hospitality during her stay. The second and third named authors would like to thank Department of Atomic Energy, Government of India [project no. 12-R\&D-TFR-5.01-0500] for the funding. Third named author acknowledges Tata Institute of Fundamental Research for the postdoctoral position and the hospitality during his stay.
	
	\section{Notation and Preliminaries} \label{sec2}
	In this section, we set up some notation and preliminaries. We refer to \cite{BK}, \cite{Hum1}, \cite{Hum2}, \cite{Jan} for preliminaries in algebraic groups and Lie algebras. 
	
	Let $G,$ $T,$ $B,$ $R,$ $R^{+},$ and $W$ be as in the introduction. Let $S=\{\alpha_{1},\ldots, \alpha_{n}\}$ denote the set of simple roots in $R^{+},$ where $n$ is the rank of $G.$ For $\beta \in R^{+},$ we use the notation $\beta>0.$ Let $W=N_{G}(T)/T$ denote the Weyl group of $G$ with respect to $T.$ The simple reflection
	in $W$ corresponding to $\alpha_{i}$ is denoted by $s_{i}$. For $w\in W,$ let $\ell(w)$ denote the length of $w.$ The Bruhat-Chevalley order $``\le"$ on $W$ is defined as for $v,w\in W,$ $v\le w$ if and only if  $X(v)\subseteq X(w).$

	For a subset $J\subset S,$ let $W_{J}$ be the subgroup of $W$ generated by $\{s_{\alpha}: \alpha \in J\}.$ For a subset $J\subseteq S,$ let $P_{J}$ be the standard parabolic subgroup of $G,$ i.e., $P_{J}$ is generated by $B$ and $n_{w},$ where $w\in W_{J}$ and $n_{w}$ is a representative of $w$ in $G.$ The subgroup $W_{J}\subseteq W$ is called Weyl group of $P_{J}.$ For a simple root $\alpha_{i},$ we denote the corresponding parabolic subgroup simply by $P_{\alpha_{i}},$ it is called minimal parabolic subgroup corresponding to $\alpha_{i}.$
	
	Let $\mathfrak{g}$ be the Lie algebra of $G.$ Let $\mathfrak{t}\subset \mathfrak{g}$ be the Lie algebra of $T$ and $\mathfrak{b}\subset \mathfrak{g}$ be the Lie algebra of $B.$ 
	Let $X(T)$ denote the group of all characters of $T.$ We have $X(T)\otimes_{\mathbb{Z}} \mathbb{R}={\rm Hom}_{\mathbb{R}}(\mathfrak{t}_{\mathbb{R}}, \mathbb{R}),$ the dual of the real form of $\mathfrak{t}.$ The positive definite $W$-invariant form on ${\rm Hom}_{\mathbb{R}}(\mathfrak{t}_{\mathbb{R}}, \mathbb{R})$ induced by the Killing form of $\mathfrak{g}$ is denoted by $(- , -).$ We use the notation $\langle -,- \rangle,$ to denote $\langle \mu, \alpha\rangle=\frac{2(\mu, \alpha)}{(\alpha, \alpha)}$ for every $\mu\in X(T)\otimes_{\mathbb{Z}} \mathbb{R}$ and $\alpha \in R.$ 
	The pairing $\langle \mu, \alpha \rangle$ is usually called the Cartan pairing of $\mu$ and $\alpha$. 
	We denote by $X(T)^{+}$ the set of dominant characters of $T$ with respect to $B^{+}.$ Let $\rho$ denote
	the half sum of all positive roots of $G$ with respect to $T$ and $B^{+}.$ For any simple root $\alpha_{i},$ we denote the fundamental weight corresponding to $\alpha_{i}$ by $\omega_{i}.$ For $1\le i \le n,$ let $h(\alpha_{i})\in \mathfrak{t}$ be the fundamental co-weight corresponding to $\alpha_{i}.$ That is $\alpha_{i}(h(\alpha_{j})) = \delta_{ij},$ where $\delta_{ij}$ is Kronecker delta.
	
	We recall that the BSDH-variety corresponding to a reduced expression $\underline{i}= (i_1, i_2, \ldots, i_r)$ of $w =s_{i_{1}}s_{i_{2}}\cdots s_{i_{r}}$ is  defined as the quotient

	\[\displaystyle
	Z(w,\underline{i})=\frac{P_{\alpha_{i_{1}}}\times P_{\alpha_{i_{2}}}\times \cdots \times  P_{\alpha_{i_{r}}}}{B\times B\times \cdots \times B},\]
	where the $r$-fold product $B\times B \times \cdots \times B$ acts on $P_{\alpha_{i_{1}}}\times  P_{\alpha_{i_{2}}}\times \cdots \times P_{\alpha_{i_{r}}}$ on the right via \[\displaystyle (p_1, p_2,\ldots , p_r)\cdot(b_1, b_2, \ldots , b_r) = (p_1\cdot b_1, b_{1}^{-1}\cdot p_2 \cdot b_2, \ldots , b^{-1}_{r-1}\cdot p_r\cdot b_r),~ p_j\in P_{\alpha_{i_{j}}}, b_j \in B\]  (see \cite[Definition 1, p.73]{Dem1}, \cite[Definition 2.2.1, p.64]{BK}). The equivalence class of $(p_1,\ldots,p_{r})$ is denoted by $[p_1,\ldots,p_r].$ There is a natural action of $P_{\alpha_{i_{1}}}$ on $Z(w,\underline{i})$ given by the left multiplication as $p\cdot[p_1,\ldots ,p_r]=[pp_1,\ldots,p_r], ~p_j\in P_{\alpha_{i_{j}}}, p\in P_{\alpha_{i_{1}}}.$ In particular there is a natural left action of $B$ on $Z(w,\underline{i})$.

	Note that $Z(w, \underline{i})$ is a smooth projective variety. The BSDH-varieties are equipped with a $B$-equivariant morphism
	\begin{equation*}
		\phi_{w}: Z(w,\underline{i})\longrightarrow G/B
	\end{equation*}
	defined by $$[p_1,...,p_r]\mapsto p_1\cdots p_rB.$$ Then $\phi_{w}$ is the natural birational surjective morphism from $Z(w, \underline{i})$ to $X(w).$

	Let $f_{r}:Z(w, \underline{i})\longrightarrow Z(ws_{i_{r}}, \underline{i}')$ denote the map induced by the projection
	\[\displaystyle P_{\alpha_{i_{1}}}\times  P_{\alpha_{i_{2}}}\times \cdots \times P_{\alpha_{i_{r}}}\longrightarrow P_{\alpha_{i_{1}}}\times P_{\alpha_{i_{2}}}\times \cdots \times P_{\alpha_{i_{r-1}}}\] where $\underline{i}'= (i_1, i_2, \ldots, i_{r-1}).$ Then we observe that $f_r$ is a $P_{\alpha_{i_{r}}}/B \simeq \mathbb{P}^{1}$-fibration.
	
	For a $B$-module $V,$ let $\mathcal{L}(w,V)$ denote the restriction of the associated homogeneous vector bundle $\mathcal{L}(V)$ on $G/B$ to $X(w).$ By abuse of notation, we denote the pull back of $\mathcal{L}(w, V)$ via $\phi_{w}$ to $Z(w, \underline{i})$ also by $\mathcal{L}(w,V),$ when there is no confusion. Since for any $B$-module $V$ the vector bundle $\mathcal{L}(w,V)$ on $Z(w, \underline{i})$ is the pull back of the homogeneous vector bundle from $X(w),$ we conclude that the cohomology modules $H^j(Z(w, \underline{i}),\mathcal{L}(w,V))\simeq H^j(X(w),\mathcal{L}(w,V))$ for all $j\ge 0$ (see \cite[Theorem 3.3.4(b)]{BK}), are independent of
	choice of reduced expression $\underline{i}.$ Hence we denote $H^j(Z(w,\underline{i}),\mathcal{L}(w,V))$ by $H^j(w,V).$ In particular, if $\lambda$ is character of $B,$ then we denote the cohomology modules $H^j(Z(w, \underline{i}),\mathcal{L}(w,\lambda))$ by $H^j(w,\lambda),$ where $\mathcal{L}(w,\lambda)=\mathcal{L}(w, \mathbb{C}_{\lambda})$ and $\mathbb{C}_{\lambda}$ denotes
	the one-dimensional $B$-module associated to $\lambda.$

We use the following ascending 1-step construction of the BSDH variety as a basic tool for computing cohomology modules. Let $\alpha$ be a simple root such that $\ell(w) = \ell(s_{\alpha} w)+1.$ Let $Z(w,\underline{i})$ be a BSDH-variety corresponding to a reduced expression $w=s_{i_{1}}s_{i_{2}}\cdots s_{i_{r}}$, where $\alpha_{i_1}=\alpha.$ Then we have a natural first factor projection morphism \[p: Z(w,\underline{i})\to P_{\alpha}/B\] with fibres $Z(s_{\alpha}w,\underline{i'}),$ where $\underline{i'}= (i_2, i_ 3,\ldots,i_r).$ Note that $p$ is $P_{\alpha}$-equivariant.
By an application of the Leray spectral sequence together with the fact that the base is $\mathbb{P}^{1}$, for every $B$-module $V$ and $j\ge 0$, we obtain the following short exact sequence of $P_{\alpha}$-modules:
\[0\to H^{1}(s_{\alpha}, R^{j-1}p_{*}\mathcal{L}(w, V))\to  H^j(w,V)\to H^0(s_{\alpha}, R^{j}p_{*}\mathcal{L}(w, V))\to 0. \hspace{.3cm} \text{(SES)}\]
Moreover by \cite[II, p.366]{Jan} we have the following isomorphism \[R^{j}{f_{r}}_{*}\mathcal{L}(w, V)\simeq \mathcal{L}(ws_{i_{r}}, H^j(s_{i_{r}}, V))~ (j\ge 0) \]
\[R^{j}p_{*}\mathcal{L}(w, V)\simeq \mathcal{L}(s_{\alpha}, H^j(s_{\alpha}w, V))~ (j\ge 0) \] of $B$-linearized sheaves.
	
	 Here, we recall the following result due to Demazure \cite[p.271]{Dem2} on short exact sequence of $B$-modules:
	
	\begin{lemma} \label{lemma 2.1}
		Let $\alpha$ be a simple root and $\lambda\in X(T)$ be such that $\langle \lambda , \alpha \rangle \ge 0.$ Let $ev:H^0(s_{\alpha}, \lambda)\longrightarrow \mathbb{C}_{\lambda}$ be the evaluation map. Then we have 
		\begin{enumerate}
			\item[(1)] If $\langle \lambda, \alpha \rangle =0,$ then $H^0(s_{\alpha}, \lambda)\simeq \mathbb{C}_{\lambda}.$
			\item[(2)] If $\langle \lambda , \alpha \rangle \ge 1,$ then $\mathbb{C}_{s_{\alpha}(\lambda)}\hookrightarrow H^0(s_{\alpha}, \lambda) $, and there is a short exact sequence of $B$-modules:
			$$0\rightarrow H^0(s_{\alpha}, \lambda-\alpha)\longrightarrow H^0(s_{\alpha}, \lambda)/\mathbb{C}_{s_{\alpha}(\lambda)}\longrightarrow \mathbb{C}_{\lambda}\rightarrow 0.$$ Furthermore, $H^{0}(s_{\alpha}, \lambda- \alpha)=0$ when $\langle\lambda , \alpha \rangle=1.$
			\item[(3)] Let $n=\langle \lambda ,\alpha \rangle.$ As a $B$-module, $H^0(s_{\alpha}, \lambda)$ has a composition series \[0\subseteq V_{n}\subseteq V_{n-1}\subseteq \dots \subseteq V_{0}=H^0(s_{\alpha},\lambda)\] such that  $V_{i}/V_{i+1}\simeq \mathbb{C}_{\lambda - i\alpha}$ for $i=0,1,\dots,n-1$ and $V_{n}=\mathbb{C}_{s_{\alpha}(\lambda)}.$
			
		\end{enumerate}
	\end{lemma}
	We define the dot action by $w\cdot \lambda= w(\lambda + \rho)-\rho.$
	Now onwards we will denote the Levi subgroup of $P_{\alpha}$ \space ($\alpha \in S$) containing $T$ by $L_{\alpha}$ and the subgroup $L_{\alpha}\cap B$ by $B_{\alpha}.$
	\begin{lemma}\label{lemma 2.3}
		Let $V$ be an irreducible  $L_{\alpha}$-module. Let $\lambda$
		be a character of $B_{\alpha}$. Then we have 
		\begin{enumerate}
			\item As $L_{\alpha}$-modules, $H^j(L_{\alpha}/B_{\alpha}, V \otimes \mathbb C_{\lambda})\simeq V \otimes
			H^j(L_{\alpha}/B_{\alpha}, \mathbb C_{\lambda})$ for every $j\ge 0.$
			\item If
			$\langle \lambda , \alpha \rangle \geq 0$, then 
			$H^{0}(L_{\alpha}/B_{\alpha} , V\otimes \mathbb{C}_{\lambda})$ 
			is isomorphic as $L_{\alpha}$-module to the tensor product of $V$ and 
			$H^{0}(L_{\alpha}/B_{\alpha} , \mathbb{C}_{\lambda})$. Further, we have 
			$H^{j}(L_{\alpha}/B_{\alpha} , V\otimes \mathbb{C}_{\lambda}) =0$ 
			for every $j\geq 1$.
			
			\item If
			$\langle \lambda , \alpha \rangle  \leq -2$, then 
			$H^{0}(L_{\alpha}/B_{\alpha} , V\otimes \mathbb{C}_{\lambda})=0$, 
			and $H^{1}(L_{\alpha}/B_{\alpha} , V\otimes \mathbb{C}_{\lambda})$
			is isomorphic to the tensor product of  $V$ and $H^{0}(L_{\alpha}/B_{\alpha} , 
			\mathbb{C}_{s_{\alpha}\cdot\lambda})$. 
			\item If $\langle \lambda , \alpha \rangle  = -1$, then 
			$H^{j}( L_{\alpha}/B_{\alpha} , V\otimes \mathbb{C}_{\lambda}) =0$ 
			for every $j\geq 0$.
		\end{enumerate}
	\end{lemma}
	
	\begin{proof} (1) follows from \cite[Proposition 4.8, 5.12, p.53, p.77, I]{Jan}.
		
	Proof of (2)--(4) follows from (1) and \cite[Proposition 5.12, p.77, I; Proposition 5.2, p. 218, II]{Jan}.
	\end{proof}
	 Let $w\in W$, $\alpha$ be a simple root, and set $v=ws_{\alpha}$. As a consequence of Lemma \ref{lemma 2.3}, we have the following lemma.
	\begin{lemma} \label{lemma 2.2}
		If $\ell(w) =\ell(v)+1$, then we have
		\begin{enumerate}
			
			\item  If $\langle \lambda , \alpha \rangle \geq 0$, then 
			$H^{j}(w , \lambda) = H^{j}(v, H^0({s_\alpha, \lambda}) )$ for all $j\geq 0$.
			
			\item  If $\langle \lambda ,\alpha \rangle \geq 0$, then $H^{j}(w , \lambda ) = H^{j+1}(w , s_{\alpha}\cdot \lambda)$ for all $j\geq 0$.
			
			\item If $\langle \lambda , \alpha \rangle  \leq -2$, then $H^{j+1}(w , \lambda ) = H^{j}(w ,s_{\alpha}\cdot \lambda)$ for all $j\geq 0$. 
			
			\item If $\langle \lambda , \alpha \rangle  = -1$, then $H^{j}( w ,\lambda)$ vanish for every $j\geq 0$.
			
		\end{enumerate}
		
	\end{lemma}
\begin{proof}
Choose a reduced expression of $w= s_{i_1}s_{i_{2}}\cdots s_{i_r}$ with $\alpha_{i_r}=\alpha.$ Hence
$v= s_{i_1}s_{i_2}\cdots s_{i_{r-1}}$ is a reduced expression for $v.$ Let $\underline{i}=(i_1,i_2,\ldots, i_r)$ and $\underline{i'} = (i_1,i_2,\ldots, i_{r-1}).$ Now consider the $P_{\alpha_{i_{r}}}/B~(\simeq \mathbb{P}^{1})$ fibration  $f_r: Z(w,\underline{i})\longrightarrow Z(v,\underline{i'})$ defined as above. Then, we have an isomorphism \[\displaystyle R^{j}{f_{r}}_{*}\mathcal{L}(w,
\lambda)\simeq \mathcal{L}(v,H^j(s_\alpha,\lambda))~\text{for}~ j\ge 0\]of $B$-linearized sheaves. 

Proof of (1): Since $\langle\lambda,\alpha \rangle\ge 0,$ by Lemma \ref{lemma 2.3}(2) it follows that $R^1{{f_{r}}_{*}}\mathcal{L}(w,\lambda)=\mathcal{L}(v,H^1(s_{\alpha}, \lambda))= 0.$ Therefore by an application of a degenerate case of the Leray spectral sequences (as in \cite[Chapter 14, Section 14.6, p.369, II]{Jan}) we have $H^j(w,\lambda)=H^j(v,H^0(s_{\alpha},\lambda))$ for all $j\ge0.$

Proof of (2): Since $\langle \lambda, \alpha\rangle \ge0,$ we have $\langle s_{\alpha}\cdot \lambda,\alpha\rangle<0.$ Hence by Lemma \ref{lemma 2.3} we have ${{f_{r}}_{*}}\mathcal{L}(w,s_{\alpha}\cdot\lambda)=\mathcal{L}(v,H^0(s_{\alpha}, s_{\alpha}\cdot\lambda))= 0.$ Therefore by an application of a degenerate case of the Leray spectral sequence we have $H^{j+1}(w, s_{\alpha}\cdot \lambda)=H^j(v, H^1(s_{\alpha}, s_{\alpha}\cdot\lambda)).$
By \cite[Theorem 1]{Dem2} we have $H^0(s_{\alpha}, \lambda)=H^1(s_{\alpha}, s_{\alpha}\cdot\lambda).$ Hence the proof follows.

Proof of (3): This case is similar to (2).

Proof of (4): Since $\langle \lambda, \alpha\rangle =-1,$ by Lemma \ref{lemma 2.3}(4) we have ${R^{j}{f_{r}}_{*}}\mathcal{L}(w,\lambda)=0$ for $j\ge 0$. Therefore by an application of a degenerate case of the Leray spectral sequence we have $H^{j}(w, \lambda)=0$ for all $j\ge 0.$
\end{proof}

	The following crucial lemma will be used to compute the cohomology modules in this paper.
	
	Let $p: \widetilde{G}\longrightarrow G$ be the universal cover. Let $\widetilde{L}_{\alpha}$ (respectively, $\widetilde{B}_{\alpha}$) be the inverse image of $L_{\alpha}$ (respectively, $B_{\alpha}$).
	Recall the structure of indecomposable $\widetilde{B}_{\alpha}$-modules (see \cite{CPS},\cite[Corollary 9.1, p.130]{BKS}).
	\begin{lemma}\label{lemma 2.4}
	Any finite dimensional indecomposable $\widetilde{B}_{\alpha}$-module $V$ is isomorphic to $V^{\prime}\otimes \mathbb{C}_{\lambda}$ for some irreducible representation $V^{\prime}$ of $\widetilde{L}_{\alpha}$ and for some character $\lambda$ of $\widetilde{B}_{\alpha}$.
	
	\end{lemma}
	\section{Automorphism group of a $G$-induced variety}
	In this section we study the connected component containing the identity automorphism of the group of all algebraic automorphisms of a $G$-induced variety.

	Let $F$ be an irreducible projective $B$-variety and $E=G\times^{B} F$ be the $G$-induced variety associated to $F.$ Consider the natural projection map \[\pi: E\longrightarrow G/B; \hspace{.3cm} [g,f]\mapsto gB.\] Then for the natural action of $G$ on $G/B,$ $\pi$ is a $G$-equivariant fibration over $G/B$ with fiber $F.$
	
	{\bf Observation:} For a $G$-induced variety $E$, if the action of $B$ on $F$ extends to an action of $G$ on $F$ then the map \[\psi: G\times F\longrightarrow G/B\times F, \hspace{.3cm} (g,f)\mapsto (gB, gf)\] induces $G$-equivariant isomorphism $G\times^{B}F \longrightarrow G/B\times F,$ where $G$ acts diagonally on $G/B \times F.$

	\begin{proposition}\label{prop 3.1}
		Then $\pi$ induces a surjective homomorphism $\pi_{*} :{\rm Aut}^0(E)\longrightarrow G$ of algebraic groups. In particular, ${\rm Aut}^0(E)=\ker\pi_{*}\rtimes G,$ where $\ker\pi_{*}$ denotes the kernel of $\pi_{*}.$
	\end{proposition}
	\begin{proof}
		Since $F$ is an irreducible projective variety, we have $\pi_{*}\mathcal{O}_{E} =\mathcal{O}_{G/B},$ where $\mathcal{O}_{E}$ and $\mathcal{O}_{F}$ denote the structure sheaf on $E$ and $F$ respectively. Therefore, by \cite[Corollary 2.2, p.45]{Bri}, $\pi$ induces an algebraic group homomorphism \[\pi_{*}: {\rm Aut}^0(E)\longrightarrow {\rm Aut}^0(G/B)\] defined as follows: \[f\mapsto \pi_{*}(f):gB\mapsto \pi(f(y)); ~\text{where}~y\in\pi^{-1}(gB).\] Since $\pi^{-1}(gB)$ is connected projective variety, by the rigidity lemma it follows that $\pi_{*}(f)$ is well defined (see \cite[Proposition 2.1., p.42]{Bri}).
		Further, since ${\rm Aut}^0(G/B)=G$ (see \cite{Dem aut}, \cite[Theorem 2, p.75]{Akh}), we have $\pi_{*}: {\rm Aut}^0(E) \longrightarrow G.$
		
		Let $\sigma: G\longrightarrow {\rm Aut}^0(E)$ be the map induced by the natural action of $G$ on $E.$ Note that $\sigma$ is not a trivial map as the action of $G$ on $E$ is effective because it descends to the effective action of $G$ on $G/B$. Thus, $\sigma :G \longrightarrow {\rm Aut}^0(E)$ is an injective homomorphism of algebraic groups. Hence, $\pi_{*}$ is a surjective homomorphism of algebraic groups. Therefore, we have ${\rm Aut}^0(E) = \ker\pi_{*}\rtimes G.$
	\end{proof}

	It would be an interesting question to ask when does there exist an isomorphism between  $E$ and  $G/B\times F?$
	
	We have already observed that if the action of $B$ on $F$ extends to an action of $G$ on $F,$ then there is a $G$-equivariant isomorphism between $E$ and $G/B \times F.$ 
	
	Here, we give another sufficient condition under which there is a $G$-equivariant isomorphism between $E$ and $G/B\times F.$ 
	\begin{proposition}
		Assume that there exists a $B$-equivariant morphism  $\Phi: E\longrightarrow F$ such that $\Phi_{*}\mathcal{O}_{E}=\mathcal{O}_{F}.$ Then we have 
		\begin{itemize}
			\item [(i)]  $E\simeq G/B\times F.$  
			
			\item [(ii)] ${\rm Aut}^0(E)=G\times {\rm Aut}^0(F).$
		\end{itemize}
	\end{proposition}
	\begin{proof}
		Proof of (i): 
		Since $\Phi_{*}\mathcal{O}_{E} =\mathcal{O}_{F},$ by \cite[Corollary 2.2, p.45]{Bri} $\Phi$ induces an algebraic group homomorphism $\Phi_{*}: {\rm Aut}^0(E)\longrightarrow {\rm Aut}^0(F).$ Note that  by Proposition \ref{prop 3.1}, $G\subset {\rm Aut}^0(E).$ Thus $G$ acts on $F$ via the map $\Phi_{*},$ i.e., the action of $G$ on $F$ is given by $g*f=\Phi(g\cdot z),$ where $g\in G, f\in F$ and $z\in \Phi^{-1}(f)$ (see \cite[Proof of Proposition 2.1, p.42]{Bri}). Further, since $\Phi$ is  $B$-equivariant, this action of $G$ on $F$ is an extension of the $B$ action on $F.$ Therefore, by the above observation we have a $G$-equivariant isomorphism $E\simeq G/B\times F$.
		
		Proof of (ii):	By using (i) and \cite[Corollary 2.3, p.46]{Bri}, we have ${\rm Aut}^0(E)= {\rm Aut}^0(G/B)\times {\rm Aut}^0(F).$ Moreover, since ${\rm Aut}^0(G/B)=G$ (see \cite{Dem aut}), we have ${\rm Aut}^0(E)= G\times {\rm Aut}^0(F).$ 
		
	\end{proof}
	
	\begin{theorem}\label{thm 3.3}
		Let $F,$ $E$ be as before. Let $\Theta_F$ (respectively, $\Theta_{E}$) be the tangent sheaf of $F$ (respectively, of $E$). Then we have
		\begin{itemize}
			\item [(i)] ${\rm Aut}^0(E)=G,$ if $H^0(G/B, H^0(F, \Theta_{F}))= 0.$
			
			\item[(ii)]	Assume that $F$ satisfies $H^j(F, \mathcal{O}_{F})=0$ for all $j\ge 1.$ Then $H^1(E, \Theta_{E})=H^0(G/B, H^1(F, \Theta_{F})),$ if $H^j(G/B,H^0(F, \Theta_{F}))=0$ for $j = 1, 2.$
		\end{itemize}
	\end{theorem}
	\begin{proof}
		Proof of (i): Recall that $\pi: E\longrightarrow G/B$ is the natural projection given by $[g, f]\mapsto gB,$ where $g\in G,$ and $f\in F.$
		
		Consider the exact sequence of $\mathcal{O}_E$-modules
		\begin{equation}\label{eq3.1}
			0\longrightarrow \mathcal{R}\longrightarrow \Theta_{E}\longrightarrow \pi^*\Theta_{G/B}\longrightarrow 0, 
		\end{equation}
		where $\mathcal{R}$ denotes the relative tangent sheaf with respect to the map $\pi.$

	Therefore, \eqref{eq3.1} induces the following long exact sequence
	\begin{equation}\label{eq3.4}
		0\rightarrow H^0(E, \mathcal{R})\rightarrow H^0(E, \Theta_{E})\rightarrow H^0(E, \pi^{*}\Theta_{G/B})\rightarrow H^1(E, \mathcal{R})\rightarrow H^1(E, \Theta_{E})\rightarrow \cdots
	\end{equation}
	of $G$-modules.
	
	Since $H^0(F, \mathcal{O}_F)=\mathbb{C}$ and $\pi$ is a projective morphism, we have 
	\begin{equation}\label{eq3.2}
	\pi_*(\pi^{*}\mathcal{O}_{G/B})=\pi_{*}\mathcal{O}_E=\mathcal{O}_{G/B}.
	\end{equation}
		Now by using projection formula (see \cite[Chapter III, Ex 8.3, p.253]{Har}) and \eqref{eq3.2}, we have 
		\begin{equation}\label{eq3.3}
			\pi_{*}(\pi^*\Theta_{G/B}) = \Theta_{G/B}\otimes \pi_{*}\mathcal{O}_{E} = \Theta_{G/B}. 	
		\end{equation}
		Further, since $H^0(G/B, \Theta_{G/B})=\mathfrak{g}$  (see \cite{Dem aut},\cite[Theorem 2, p.75 and Theorem 1, p.130]{Akh}), we have $H^0(E, \pi^*\Theta_{G/B})=\mathfrak{g}.$

On the other hand, by Proposition \ref{prop 3.1}, we see that $\sigma: G\longrightarrow {\rm Aut}^0(E)$ is an injective homomorphism of algebraic groups. Since Lie(${\rm Aut}^0(E)$) = $H^0(E,\Theta_{E})$
	(see \cite[Lemma 3.4, p.13]{MO}), the differential $d\sigma :\mathfrak{g}\longrightarrow H^0(E, \Theta_{E})$ is an injective
	homomorphism of Lie algebras.
		
Therefore, \eqref{eq3.4} gives the following short exact sequence
\begin{equation}\label{eq3.5}
0\longrightarrow H^0(E, \mathcal{R})\longrightarrow H^0(E, \Theta_{E})\longrightarrow H^0(E, \pi^{*}\Theta_{G/B})\longrightarrow 0
\end{equation}
of $G$-modules.

Now, since the restriction of $\mathcal{R}$ to $F$ coincides with the tangent sheaf $\Theta_{F}$ of $F,$ it follows that $H^0(E, \mathcal{R})=H^0(G/B, H^0(F, \Theta_{F})).$ Thus, we have  $H^0(E, \mathcal{R})= 0,$ as $H^0(G/B, H^0(F, \Theta_{F} )) = 0.$
		Therefore, by using \eqref{eq3.5}, we have $H^0(E, \Theta_E)=\mathfrak{g},$ as $H^0(E, \pi^*\Theta_{G/B}) = \mathfrak{g}.$  Hence, ${\rm Aut}^0(E)=G.$ 
		
		Proof of (ii): 	Since $H^j(F,\mathcal{O}_F)= 0$ for $j \ge 1,$ we have
		\begin{equation}\label{eq3.6}
			R^j\pi_{*}(\pi^{*}\mathcal{O}_{G/B})=R^j\pi_{*}\mathcal{O}_{E}= 0 \text{~for~} j\ge 1.
		\end{equation}
	Therefore, by using projection formula (see \cite[Chapter III, Ex 8.3, p.253]{Har}) and \eqref{eq3.6}, we have 
	\begin{equation}\label{eq3.7}
		R^j\pi_{*}(\pi^*\Theta_{G/B})= \Theta_{G/B}\otimes R^j\pi_{*}(\mathcal{O}_E)= 0 \text{~for all~} j\ge 1. 	
	\end{equation}
		The $E_{2}^{i,j}$ term of Leray spectral sequence for $\pi$ and $\pi^{*}\Theta_{G/B}$ is
		\begin{equation}
			E_{2}^{i,j}= H^i(G/B, R^j\pi_{*}(\pi^{*}\Theta_{G/B})).
		\end{equation}	
		Since $R^j\pi_{*}(\pi^{*}\Theta_{G/B})=0$ for all $j\ge 1$ (see \eqref{eq3.7}), we have $E_{2}^{i,j}= 0$ for $j\ge 1.$ Therefore, by using degenerate case of Leray spectral sequence and \eqref{eq3.3}, we have $$H^j(E, \pi^*\Theta_{G/B})=H^j(G/B, \pi_{*}(\pi^*\Theta_{G/B}))=H^j(G/B, \Theta_{G/B})$$ for $j\ge 1.$
		
Now, since $H^j(G/B, \Theta_{G/B})=0$ for all $j\ge 1$ (see \cite{Dem aut},\cite[Theorem 2, p.75 and Theorem 1, p.130]{Akh}), we have $H^j(E,\pi^*\Theta_{G/B}) = 0$ for $j\ge 1.$ Therefore, \eqref{eq3.4} induces the following exact sequence
			\begin{equation*}
				0\rightarrow H^0(E, \mathcal{R})\rightarrow H^0(E, \Theta_{E})\rightarrow H^0(E, \pi^{*}\Theta_{G/B})\rightarrow H^1(E, \mathcal{R})\rightarrow H^1(E, \Theta_{E})\rightarrow 0
			\end{equation*}
of $G$-modules and
			\begin{equation}\label{eq3.9}
				H^j(E, \mathcal{R})\simeq H^j(E, \Theta_{E}) \text{~for~} j \ge 2. 
			\end{equation}	
Moreover, by using \eqref{eq3.5} and \eqref{eq3.9}, we have
		
		\begin{equation}
			H^j(E, \mathcal{R})\simeq H^j(E, \Theta_{E}) \text{~for~} j\ge 1.
		\end{equation}
Now, since $H^j(G/B, H^0(F, \Theta_F )) = 0$ for $j = 1, 2,$ by using the five term exact sequence associated to the spectral sequence, we have $H^1(E, \Theta_{E})= H^0(G/B, H^1(F, \Theta_{F})).$ 
	\end{proof}
	
	\begin{corollary}\label{cor 3.4}
		Let $F,$ $E$ be as in Theorem \ref{thm 3.3} and $F$ satisfies $H^0(G/B, H^0(F,\Theta_{F}))=0$ but $H^0(F, \Theta_{F})\neq 0.$ Then $E$ is not isomorphic to $G/B\times F.$ In particular, the action of $B$
		on $F,$ cannot be extended to an action of $G$ on $F.$
	\end{corollary} 
	\begin{proof}
		If $E\simeq G/B\times F,$ then by \cite[Corollary 2.3, p.46]{Bri},  ${\rm Aut}^0(E)={\rm Aut}^0(G/B)\times {\rm Aut}^0(F).$ Since ${\rm Aut}^0(G/B)=G$ (see \cite{Dem aut}), we have ${\rm Aut}^0(E)= G \times {\rm Aut}^0(F).$ Further, since $H^0(F, \Theta_{F})\neq 0,$ by \cite[Lemma 3.4, p.13]{MO}, we conclude that ${\rm Aut}^0(F)$ is not a trivial group. Therefore, ${\rm Aut}^0(E)\neq G,$ which shows contradiction to Theorem \ref{thm 3.3}(i).
	\end{proof}

	{Note:}\label{Note} For any $G$-induced variety $E,$ we have the following observations from the proof of Theorem \ref{thm 3.3}.

	\begin{itemize}
		\item[(1)] Let $\pi_{*}: {\rm Aut}^0(E)\longrightarrow G$ be as in Proposition \ref{prop 3.1}. Then we have Lie($\ker\pi_*)=H^0(E,\mathcal{R})=H^0(G/B,H^0(F,\Theta_{F})).$	
	\end{itemize}
	\begin{itemize}
		\item[(2)]   Lie(${\rm Aut}^0(E)$) fits into the exact sequence 
		\begin{equation*}
			0\rightarrow H^0(E,\mathcal{R})\rightarrow H^0(E, \Theta_{E})\rightarrow \mathfrak{g} \rightarrow 0
		\end{equation*}
		of $G$-modules. Since there is an embedding $\mathfrak{g}\hookrightarrow H^0(E, \Theta_{E})$ coming from the faithful action of $G$ on $E$, the above exact sequence splits, i.e., $H^0(E, \Theta_{E})=H^0(E, \mathcal{R})\oplus \mathfrak{g}$ as $G$-modules. 
	\end{itemize}

	\section{$G$-Schubert variety and $G$-BSDH-variety}
	Throughout this section we assume $G$ to be a simply-laced simple algebraic group of adjoint type. In this section we study the connected component containing the identity automorphism of the group of all algebraic automorphisms of a $G$-Schubert variety and $G$-BSDH-variety.
	
	\subsection{Identity component of the automorphism group of a $G$-Schubert variety:}
	We recall some results on automorphism group of a Schubert variety from \cite{Kan}. 
	
	Recall that for $w$ in $W$ the Schubert variety in $G/B$ associated to $w$ is usually denoted by $X(w)$ defined as \[ X(w):=\overline{BwB/B}\subset G/B.\]
	
	For the left action of $G$ on $G/B,$ let $P={\text{Stab}}_{G}(X(w))$ denote the stabilizer of $X(w)$ in $G.$ Since $B\subset {\rm Stab}_{G}(X(w))$, $P$ is a parabolic subgroup of $G$. Further since $P$ contains $B,$ it is a standard parabolic subgroup of $G$ of the form $P_{I(w)}$ for some subset $I(w)$ of $S$. This subset $I(w)$ of $S$ is precisely consisting of $\alpha\in S$ such that $w^{-1}(\alpha)<0$, i.e., $w^{-1}(\alpha)$ is a negative root.  
	Since $P_{I(w)}$ is connected, the natural left action of $P_{I(w)}$ on $X(w),$ induces a map \[\varphi_{w}: P_{I(w)}\longrightarrow {\rm Aut}^0(X(w)).\] Let $\alpha_{0}$ denote the highest root of $G$ with respect to $T$ and $B^{+}.$  Let $\mathfrak{p}_{I(w)}$ denote the parabolic Lie subalgebra of $\mathfrak{g}$. Then $\mathfrak{p}_{I(w)}$ is the Lie algebra of $P_{I(w)}$.
	\begin{theorem}\label{Thm 4.1}
	The map $\varphi_{w}$ is a surjective homomorphism of algebraic groups.
	\end{theorem} 
  Theorem \ref{Thm 4.1} is stated in \cite[Theorem 4.2(1), p.772]{Kan} for a smooth Schubert variety, but proof goes for any Schubert variety. Here we give a brief sketch of the proof.
	\begin{proof}
		Recall from \cite[Lemma 3.4, p.13]{MO} that Lie(${\rm Aut}^0(X(w))$)= $H^0(X(w),\Theta_{X(w)}).$
		To prove $\varphi_{w}$ is surjective it is enough to prove that $d\varphi_{w}: \mathfrak{p}_{I(w)}\longrightarrow H^0(X(w), \Theta_{X(w)})$ is
		surjective.
		
		Let $\Theta_{G/B}$ be the tangent sheaf of $G/B.$ Then note that $\Theta_{G/B}$ is the sheaf corresponding to the tangent bundle $\mathcal{L}(\mathfrak{g/b})$ of $G/B.$ Further, we
		have $H^0(X(w), \Theta_{X(w)})\subseteq H^0(X(w), \Theta_{G/B}|_{X(w)})= H^0(w,\mathfrak{g/b}).$
		
		By \cite[Lemma 3.5, p.770]{Kan}, the restriction map $H^0(G/B, \mathfrak{g/b})\longrightarrow H^0(w,\mathfrak{g/b})$ is surjective.
		Thus for $D'\in H^0(X(w), \Theta_{X(w)})\subseteq H^0(w, \Theta_{G/B}|_{X(w)})= H^0(w, \mathfrak{g/b}),$ there exists $D\in H^0(G/B, \Theta_{G/B})$ such that image under the restriction map is $D'.$  Consequently, $D$ preserves the ideal sheaf of $X(w)$ in $G/B,$ and hence $D\in$ Lie$(\text{Stab}_G(X(w)))=\mathfrak{p}_{I(w)}.$ Therefore, proof of the lemma follows.
	\end{proof} 
	
	We recall some definitions and facts which we use later (see \cite{Bot},\cite{Akh},\cite{Snow}).

	Let $\lambda \in X(T).$ Then $\lambda$ is called singular if $\langle \lambda, \alpha \rangle = 0$ for some $\alpha \in R^{+},$ otherwise it is called non-singular. 
	
	The index of $\lambda$ is defined to be ind$(\lambda):=min\{\ell(w)|w(\lambda)\in X(T)^{+}\}.$ 
	
	Fact 1. If $\beta\in R$ is such that $\beta+\rho$ is non-singular, then either $\beta=\alpha_{0}$ or $\beta$ is the negative of a simple root.
	
	Fact 2. If $\beta \in R$ is such that $\beta+\rho$ is non-singular, then index of $\beta+\rho$ is either $0$ or $1$ (see \cite[p.47-48]{Snow}). Further, if the index of $\beta+\rho$ is $0$ (respectively,
	1), then $\beta= \alpha_{0}$ (respectively, $\beta$ is the negative of a simple root).
     
     We use the following version of Bott's theorem on vanishing of cohomology of homogeneous vector bundles, a proof of whose can be found in \cite[Theorem 1, p.129]{Gri}

   Let $P$ be a parabolic subgroup of $G$ containing $B$. Let $V$ be a $P$-module. Then $V$ has a filtration \[0=V_{0}\subset V_1\subset V_2\subset \ldots \subset V_{t}=V\] by $P$-submodules such that $V_{i}/V_{i-1}$ are irreducible $P$-modules of highest weight $\lambda_i$ ($1\le i\le t$). We call these weights the highest weights of $V$ and denote the set of weights by $\Lambda_{P}(V):=\{\lambda_i: 1\le i\le t\}$. We note here that although the filtrations are not unique but the set $\Lambda_{P}(V)$ of weights is uniquely determined by the decomposition of $V$ into direct sum of irreducible components with respect to $L$-modules, where $L$ denotes the Levi factor of $P$. Let $I_{P}(V):=\{\text{ind}(\lambda_i+\rho): \lambda_i+\rho ~\text{is non-singular~ and~} \lambda_i ~\in \Lambda_{P}(V)\}$.
    \begin{lemma}\label{lem 4.4}
   Let $\mathcal{L}_{P}(V)$ be the homogeneous vector bundle on $G/P$ associated to a $P$-module $V$.
     Then we have the following: \[H^{j}(G/P,\mathcal{L}_{P}(V))=0 ~\text{if }~ j\notin I_{P}(V).\]
    \end{lemma}
    Here we recall the description of the kernel of the map $\varphi_{w}$ from \cite[Corollary 4.3, p.774]{Kan}.
    Let ${\rm Supp}(w):=\{\alpha\in S: s_{\alpha}\le w\}$, and let $T(w)=\bigcap_{\alpha\in {\rm Supp}(w)}\ker(\alpha)$. For $w\in W,$ let $R^{+}(w^{-1}):=\{\beta\in R^{+}: w^{-1}(\beta)<0\}$. For $\beta \in R$, let $U_{\beta}$ denote the root subgroup of $G$ associated to $\beta$. Let $U_{\le w}$ be the root subgroup of $U$ (the unipotent radical of $B$) generated by \[\langle U_{-\beta}:\beta\in R^{+}\setminus \bigcup_{v\le w}R^{+}(v^{-1})\big\rangle\]

    Let $K_{w}:=\ker\varphi_{w}$. Let $\mathfrak{k}_{w}$ denote the Lie algebra of $K_{w}$. Then
    \begin{lemma}\label{lemma 4.4}
    $K_{w}$ $($ respectively, $\mathfrak{k}_{w})$ is generated by $T(w)$ $($ respectively, ${\rm Lie}(T(w)))$ and $U_{\le w}$ $($respectively, ${\rm Lie}(U_{\le w})).$
    \end{lemma}
   
   \begin{lemma}\label{lem 4.3}
    Assume that $w\in W$ is such that $w\neq w_0.$ Then $H^j(G/B, \mathfrak{p}_{I(w)})=0$ for all $j\ge 0$.
    \end{lemma}
    \begin{proof} Consider the short exact sequence
    	\[0\to \mathfrak{b}\to \mathfrak{p}_{I(w)}\to \mathfrak{p}_{I(w)}/\mathfrak{b}\to 0\]  of $B$-modules. Therefore we have the following long exact sequence 
    	\[0\to H^0(G/B,\mathfrak{b})\to H^0(G/B,\mathfrak{p}_{I(w)})\to H^0(G/B, \mathfrak{p}_{I(w)}/\mathfrak{b})\to \]\[ H^1(G/B,\mathfrak{b})\to H^1(G/B,\mathfrak{p}_{I(w)})\to H^1(G/B, \mathfrak{p}_{I(w)}/\mathfrak{b})\to \cdots\]
    	of $G$-modules.
    By using \cite[Lemma 3.4, Theorem 4.1, p.770-771]{Kan}, we have $H^j(G/B, \mathfrak{b})=0$ for $j\ge 0$. Thus by the above exact sequence we have the following:
    \[H^j(G/B,\mathfrak{p}_{I(w)})=H^j(G/B,\mathfrak{p}_{I(w)}/\mathfrak{b})~ \text{ for all}~j\ge 0.\] Note that $\mathfrak{p}_{I(w)}/\mathfrak{b}$ is a $B$-module such that the weights appearing in $\mathfrak{p}_{I(w)}/\mathfrak{b}$ are positive roots. Further since $\mathfrak{p}_{I(w)}\neq\mathfrak{g}$, as $w\neq w_{0}$, the highest root $\alpha_{0}$ does not appear in $\mathfrak{p}_{I(w)}/\mathfrak{b}$. Therefore, by using Fact 2 and Lemma \ref{lem 4.4} proof follows.
    \end{proof}

	Now we prove
	\begin{lemma}\label{lemma 4.2}
		Assume that $w\neq w_{0} \in W.$ Then we have $H^j(G/B, H^0(X(w), \Theta_{X(w)}))=0$
		for all $j\ge 0.$
	\end{lemma}
	\begin{proof}
	By Theorem \ref{Thm 4.1}, we have the following short exact sequence
		\begin{equation}\label{eq4.1}
			0\longrightarrow \mathfrak{k}_{w}\longrightarrow \mathfrak{p}_{I(w)}\longrightarrow H^0(X(w), \Theta_{X(w)})\longrightarrow 0
		\end{equation}
		of $B$-modules.
		
	   Since $w\neq w_{0},$ by Lemma \ref{lem 4.3} we have $H^j(G/B, \mathfrak{p}_{I(w)})=0$ for $j\ge 0.$ Therefore, from the long exact sequence associated to \eqref{eq4.1} we have the following: \[H^j(G/B, H^0(X(w), \Theta_{X(w)}))=H^{j+1}(G/B, \mathfrak{k}_{w})~\text{for}~ j\ge 0.\]
		
		By Lemma \ref{lemma 4.4} we have \[\mathfrak{k}_{w}=\text{Lie}(T(w))\oplus \bigoplus\limits_{\beta\in A}\mathbb{C}_{-\beta} .\] where $A=R^{+} \setminus \bigcup_{v\le w} R^{+}(v^{-1}).$

		Assume that $\text{Supp}(w)=S,$ then $T(w)$ is a trivial group and by the above description $(\mathfrak{k}_{w})_{-\alpha}=0$ for all $\alpha\in S$. Therefore, by using Fact 2 and Lemma \ref{lem 4.4} we conclude that $H^{j}(G/B, \mathfrak{k}_{w})=0$ for $j\ge 0$. Hence $H^j(G/B, H^0(X(w),\Theta_{X(w)}))=0$ for all $j\ge 0$.

		Assume that $\text{Supp}(w)\subset S$. Define $J^{c}:=S\setminus \text{Supp}(w)$. Note that $J^{c}\subset A.$ Let $L_{J^c}$ be the Levi subgroup of $P_{J^c}$, and $B_{J^c}=B\cap L_{J^c}$. Then we have $P_{J^c}/B \simeq L_{J^c}/B_{J^c}.$  Note that $P_{J^c}/B=X(w_{0,J^c})$, where $w_{0,J^c}$ denotes the longest element of $W_{J^{c}}$. Let $r=|J^c|$.  Let $c=s_{i_{r}}\cdots s_{i_{1}}$ be a reduced expression such that $\alpha_{i_{j}}\neq \alpha_{i_{k}}$ for $1\le j\ne k\le r$, i.e., $c$ is a Coxter element of $W_{J^c}$. Now extend this reduced expression of $c$ to a reduced expression $w_{0,J^c}=s_{i_{N}}\cdots s_{i_{r+1}}s_{i_{r}}s_{i_{r-1}}\cdots s_{i_{1}}$ to compute \[H^j(L_{J^c}/B_{J^c},\mathfrak{k}_{w})~\text{ for}~ j\ge 0,\] where $\ell(w_{0,J^c})=N$. 
	    
	    Note that since $G$ is simply-laced, we have $\langle -\beta,\alpha\rangle=-1,0, ~\text{or}~ 1$ for any pair $\alpha, \beta$ in $R$ such that $\alpha\neq \pm \beta.$ Hence if $\langle-\beta,\alpha\rangle=1$, then $-\beta+\alpha$ is not a root. Similarly, if $\langle-\beta,\alpha\rangle=-1$, then $-\beta-\alpha$ is not a root.
		Note that if  $\beta$ in $A$ is such that $\langle -\beta, \alpha_{i_{1}}\rangle=-1$, then clearly $\beta-\alpha_{i_{1}}$ is a positive root. 
		
		Therefore by the above description of $\mathfrak{k}_{w}$, the indecomposable $\widetilde{B}_{\alpha_{i_{1}}}$-summands of $\mathfrak{k}_{w}$ are the following: \[ \mathbb{C}h(\alpha_{i_1})\oplus \mathbb{C}_{-\alpha_{i_1}} ; \mathbb{C}h(\alpha) ~(\alpha\neq \alpha_{i_{1}})~; \mathbb{C}_{-\beta}~(\text{ for}~ \langle \beta, \alpha_{i_1}\rangle=0 );\] \[~\mathbb{C}_{-\beta}\oplus \mathbb{C}_{-\beta-\alpha_{i_{1}}}~(\text{ for }~\langle -\beta, \alpha_{i_1}\rangle=1); \] \[ \mathbb{C}_{-\beta}~ (\text{ for} \langle -\beta, \alpha_{i_1}\rangle=-1 ~\text{such that}~ \beta-\alpha_{i_1}\notin A).\]

		By Lemma \ref{lemma 2.4} we have the following: \[ \mathbb{C}h(\alpha_{i_1})\oplus \mathbb{C}_{-\alpha_{i_1}}=V(1)\otimes \mathbb{C}_{-\omega_{i_{1}}};\]
		\[\mathbb{C}h(\alpha)=V(0) ~(\alpha\neq \alpha_{i_{1}});\]
		\[\mathbb{C}_{-\beta}=V(0) ~(\text{ for }~\langle \beta, \alpha_{i_1}\rangle=0 );\] \[~\mathbb{C}_{-\beta}\oplus \mathbb{C}_{-\beta-\alpha_{i_{1}}}=V(1)~(\text{ for } \langle -\beta, \alpha_{i_1}\rangle=1); \] \[~\mathbb{C}_{-\beta}=V(0)\otimes \mathbb{C}_{-\omega_{i_{1}}}~(\text{ for } \langle -\beta, \alpha_{i_1}\rangle=-1 \text{~such that } \beta-\alpha_{i_{1}}\notin A);\] 
		where $V(i)$ denotes an $i+1$-dimensional irreducible $\widetilde{L}_{\alpha_{i_{1}}}$-module. Therefore by using Lemma \ref{lemma 2.3} we have the following:
		\[H^0(s_{i_{1}},\mathfrak{k}_{w})=\bigoplus\limits_{\alpha\in J^{c}\setminus \{\alpha_{i_{1}}\}}\mathbb{C}h(\alpha) \bigoplus\limits_{\begin{matrix}
			\beta \in A\setminus\{\alpha_{i_{1}}\}\\
			\text{except those with} \\ \langle \beta,\alpha_{i_1} \rangle=1 ~\text{and }~ \beta-\alpha_{i_{1}}\notin A
			\end{matrix} }\mathbb{C}_{-\beta}\]  and
		 	\[H^j(s_{i_{1}},\mathfrak{k}_{w})=0~  \text{for all}~ j\ge 1.\]
	Similarly, by proceeding recursively for the string $s_{i_{r}}\cdots s_{i_{2}}$, we conclude that zero weights, and negatives of simple roots do not occur in $H^0(s_{i_{r}}\cdots s_{i_{1}},\mathfrak{k}_{w})$ and $H^j(s_{i_{r}}\cdots s_{i_{1}},\mathfrak{k}_{w})=0$ for all $j\ge 1.$
       
       Therefore proceeding successively for the string $s_{i_{N}}\cdots s_{i_{r+1}}$, we conclude that zero weights, and negatives of simple roots do not occur in $H^0(L_{J^c}/B_{J^c},\mathfrak{k}_{w}),$ and $H^j(L_{J^c}/B_{J^c}, \mathfrak{k}_{w})=0$ for all $j\ge 1.$

		Consider the natural projection map \[p: G/B \longrightarrow G/P_{J^c}.\] Since $H^j(P_{J^c}/B,\mathfrak{k}_{w})=0$ for all $j\ge 1$,  $R^{j}p_{*}\mathcal{L}(\mathfrak{k}_{w})=0$ for $j\ge 1$. Therefore by an application of a degenerate case of the Leray spectral sequence we have \[H^j(G/B, \mathfrak{k}_{w})=H^j(G/P_{J^c}, H^0(P_{J^c}/B,\mathfrak{k}_{w}))~\text{for ~all}~ j\ge 1.\]
		
		Since $H^0(P_{J^c}/B, \mathfrak{k}_{w})$ is a $B$-module whose weights are among the roots other than negative of simple roots, by using Fact 2, and Lemma \ref{lem 4.4} we have $H^j(G/P_{J^c}, H^0(P_{J^c}/B, \mathfrak{k}_{w}))= 0$ for $j\ge 1.$ Therefore, $H^j(G/B, H^0(X(w), \Theta_{X(w)})) = 0$ for all $j\ge 0.$
	\end{proof}
	
	\begin{proposition}\label{prop 4.3}
		Assume that $w\neq w_{0} \in W.$ Then we have
		\begin{itemize}
			\item [(i)] ${\rm Aut}^0(\mathcal{X}(w))=G.$
			
			\item[(ii)] $H^1(\mathcal{X}(w), \Theta_{\mathcal{X}(w)})=H^0(G/B, H^1(X(w), \Theta_{X(w)})).$
		\end{itemize}
	\end{proposition}
	\begin{proof}
		Proof of (i): By Lemma \ref{lemma 4.2}, we have $H^0(G/B, H^0(X(w), \Theta_{X(w)})) = 0.$ Therefore, by Theorem \ref{thm 3.3}(i), we have ${\rm Aut}^0(\mathcal{X}(w))= G.$
		
		Proof of (ii): By Lemma \ref{lemma 4.2}, we have $H^j(G/B, H^0(X(w), \Theta_{X(w)}))= 0$ for all $j\ge 1$. By using \cite[Theorem 3.1.1(a), p.84]{BK} it follows that $H^j(X(w),\mathcal{O}_{X(w)})=0$ for all $j>0.$ Therefore, by Theorem \ref{thm 3.3}(ii), we have $H^1(\mathcal{X}(w), \Theta_{\mathcal{X}(w)})= H^0(G/B, H^1(X(w), \Theta_{X(w)})).$
	\end{proof}
	
	\begin{corollary}
		Let $w\in W$ be such that $w\neq id, w_{0},$ where $id$ denotes the identity element of $W.$ Then $\mathcal{X}(w)$ is not isomorphic to $G/B \times X(w).$
	\end{corollary}
	\begin{proof}
		If $\mathcal{X}(w)=G/B \times X(w),$ then by \cite[Corollary 2.3, p.46]{Bri}, ${\rm Aut}^0(\mathcal{X}(w))={\rm Aut}^0(G/B)\times {\rm Aut}^0(X(w)).$ Since ${\rm Aut}^0(G/B)= G$ (see \cite{Dem aut}), we have ${\rm Aut}^0(\mathcal{X}(w))= G\times {\rm Aut}^0(X(w)).$ Further, since $w\neq id,$  $X(w)$ contains an open $B$-orbit of positive dimension, whence $B$ acts non-trivially on $X(w).$ Therefore, we conclude that ${\rm Aut}^0(X(w))$ is not a trivial group. Hence, ${\rm Aut}^0(\mathcal{X}(w))\neq G,$ which shows contradiction to Proposition \ref{prop 4.3}(i) as $w\neq w_{0}.$ 
	\end{proof}
	
	\begin{remark}
		If $w= w_{0},$ then we have $\mathcal{X}(w)= G/B\times G/B.$ Thus ${\rm Aut}^0(\mathcal{X}(w))=G\times G$ and $H^j(\mathcal{X}(w), \Theta_{\mathcal{X}(w)}) = 0$ $( j\ge 1)$ for $w=w_{0}.$ 
	\end{remark}
    
	\subsection{Identity component of the automorphism group of a $G$-BSDH variety:} 
	Let $w= s_{i_{1}}s_{i_{2}}\cdots s_{i_{r}}$ be a reduced expression and $\underline{i}:=(i_1,\ldots, i_r).$ Let $v= s_{i_{1}}s_{i_{2}}\cdots s_{i_{r-1}}$ and $\underline{i}':=(i_1,\ldots, i_{r-1}).$ 
	
	Consider the $P_{\alpha_{i_{r}}}/B$ ($\simeq \mathbb{P}^{1}$) fibration ~\[\pi_{r}: G/B\to G/P_{\alpha_{i_{r}}}.\] Then the fibre product $Z(w,\underline{i})$ is of $Z(v,\underline{i'})$ and $G/B$ over $G/P_{\alpha_{i_{r}}}$ with respect to the maps $\pi_{r}$ and $\pi_{r}\circ\phi_{v}: Z(v,\underline{i})\to G/P_{\alpha_{i_{r}}}$. Thus we have the following $B$-equivariant fibre product diagram: 	\[\begin{tikzcd}
	&Z(w,\underline{i}):=Z(v,\underline{i}')\times_{G/P_{\alpha_{i_{r}}}}G/B \arrow[r, "\phi_{w}"] \arrow[d,"f_{r}"'] &
	G/B  \arrow[d, "\pi_{r}"'] &&
	\\
	&Z(v,\underline{i}') \arrow[r, "\pi_{r}\circ\phi_{v}"]&
	G/P_{\alpha_{i_{r}}}\end{tikzcd}.\]
	
Notice that the $P_{\alpha_{i_{r}}}/B$-fibration $f_{r}: Z(w,\underline{i})\to Z(v,\underline{i'})$ is already defined earlier in Section \ref{sec2}. Note that the relative tangent bundle on $G/B$ with respect to $\pi_{r}$ is the line bundle $\mathcal{L}(\alpha_{i_{r}}).$ Therefore from the above $B$-equivariant fibre product diagram we have the following short exact sequence \[0\to \mathcal{L}(w,\alpha_{i_{r}})\to \Theta_{Z(w,\underline{i})}\to f_{r}^{*}\Theta_{Z(v,\underline{i'})}\to 0\] of $B$-equivariant vector bundles on $Z(w,\underline{i})$. Thus we have the following long exact sequence of \[0\to H^0(w,\alpha_{i_{r}})\to H^0(Z(w,\underline{i}), \Theta_{Z(w,\underline{i})})\to H^0(Z(v,\underline{i'}),\Theta_{Z(v,\underline{i'})})\to \]\[H^1(w,\alpha_{i_{r}})\to H^1(Z(w,\underline{i}), \Theta_{Z(w,\underline{i})})\to H^1(Z(v,\underline{i'}),\Theta_{Z(v,\underline{i'})})\to \cdots\] of $B$-modules.

\begin{lemma}\label{lemma 4.6}
Let $(w, \underline{i})$, $(v, \underline{i'})$ be as above. Then we have an exact sequence \[0\to H^0(w,\alpha_{i_{r}})\to H^0(Z(w,\underline{i}),\Theta_{Z(w,\underline{i})})\to H^0(Z(v,\underline{i'}),\Theta_{Z(v,\underline{i'})})\to 0\] of $B$-modules. Further $H^j(Z(w,\underline{i}), \Theta_{Z(w,\underline{i})})=0$ for all $j\ge 1.$
\end{lemma}
\begin{proof}
By \cite[Corollary 3.6, p.771]{Kan} it follows that $H^j(w,\alpha_{i_{r}})=0$ for all $j\ge 1$. Therefore by using the above long exact sequence a proof of the first part follows. Proof of the second part follows by induction on $r$ using the above long exact sequence and  \cite[Corollary 3.6, p.771]{Kan} (see \cite[Proposition 3.1(2), p.673]{CKP}).
\end{proof}

\begin{lemma}\label{lemma 4.10}
Let $w=s_{i_1}s_{i_{2}}\cdots s_{i_{r}}$ be a reduced expression and $\alpha_{i_{r}}=\alpha.$ Then we have the following: 
\begin{itemize}
	\item [(1)] If $\langle \alpha, \alpha_{i_{k}}\rangle=0$ for all $1\le k\le r-1,$ then the weights of the $B$-module $H^0(w,\alpha)$ are $\alpha,$ $0,$ $-\alpha$ with multiplicity one.
	
	\item [(2)] If $\alpha_{i_{k}}\neq \alpha~ (1\le k\le r-1),$ but there exists an integer $1\le k\le r-1$ such that $\langle \alpha,\alpha_{i_{k}}\rangle=-1,$ then the weights of the $B$-module $H^0(w,\alpha)$ are $0,$ $-\alpha,$ or among the negative roots other than negatives of simple roots. Further, the multiplicity of each weight is one.
	
	\item [(3)] If $\alpha_{i_{k}}= \alpha$ for some $1\le k\le r-2,$  then the weights of the $B$-module $H^0(w,\alpha)$ are among the  negative roots other than negatives of simple roots with multiplicity one.
\end{itemize}
\end{lemma}
\begin{proof}
	
 Proof of (1): By using Borel-Weil-Bott we have \[H^0(s_{i_{r}},\alpha)=\mathbb{C}_{\alpha}\oplus\mathbb{C}_{0}\oplus \mathbb{C}_{-\alpha}.\]
Since $\langle \alpha, \alpha_{i_{k}}\rangle=0$ for all $1\le k\le r-1,$ by using SES, Lemma \ref{lemma 2.4}, and Lemma \ref{lemma 2.3} we have \[H^0(w,\alpha)=\mathbb{C}_{\alpha}\oplus\mathbb{C}_{0}\oplus \mathbb{C}_{-\alpha}.\] Therefore (1) follows.

Proof of (2): Assume that $1\le m\le r-1$ is the largest integer such that $\langle\alpha, \alpha_{i_m} \rangle=-1.$ Then by the assumption $\langle \alpha, \alpha_{i_k} \rangle=0$ for all $m+1 \le k\le r-1.$
Therefore by (1) we have \[H^0(s_{i_{m+1}}\cdots s_{i_{r}},\alpha)=\mathbb{C}_{\alpha}\oplus \mathbb{C}_{0}\oplus \mathbb{C}_{-\alpha}.\] 
By using Lemma \ref{lemma 2.4} the indecomposable $\widetilde{B}_{\alpha_{i_{m}}}$-summands of $H^0(s_{i_{m+1}}\cdots s_{i_{r}},\alpha)$ are the following \[\mathbb{C}_{\alpha}=V(0)\otimes \mathbb{C}_{-\omega_{i_{m}}};\]\[\mathbb{C}_{0}=V(0);\] \[\mathbb{C}_{-\alpha}=V(0)\otimes \mathbb{C}_{\omega_{i_{m}}}.\] 
Then by using SES and Lemma \ref{lemma 2.3} we have  \[H^0(s_{i_{m}}\cdots s_{i_{r}},\alpha)=\mathbb{C}_{0}\oplus \mathbb{C}_{-\alpha}\oplus \mathbb{C}_{-\alpha-\alpha_{i_{m}}}.\]  
Further by proceeding recursively using SES, Lemma \ref{lemma 2.4}, and  Lemma \ref{lemma 2.3} we conclude that the weights of the $B$-module $H^0(w,\alpha)$ are $0,$ $-\alpha$, or among the negative roots other than negatives of simple roots. Moreover the multiplicity of each weight is one.

Proof of (3): Assume that $1\le m\le r-2$ is the largest integer such that $\alpha_{i_{m}}=\alpha.$ Since $w=s_{i_{1}}\cdots s_{i_{r}}$ is a reduced expression, there is an integer $m+1\le t\le r-1$ such that $\langle \alpha, \alpha_{i_{t}} \rangle=-1.$
Therefore by  (2) the weights of the $B$-module $H^0(s_{i_{m+1}}\cdots s_{i_{r}},\alpha)$ are $0,$ $-\alpha,$ or among the negative roots other than negatives of simple roots with the multiplicity of each weight is one. Then by using Lemma \ref{lemma 2.4} the indecomposable $\widetilde{B}_{\alpha_{i_{m}}}$-summands of $H^0(s_{i_{m+1}}\cdots s_{i_{r}},\alpha)$ are either of the following forms
\[\mathbb{C}_{0}\oplus \mathbb{C}_{-\alpha}=V(1)\otimes \mathbb{C}_{-\omega_{i_{m}}};\] \[\mathbb{C}_{-\beta}\oplus \mathbb{C}_{-\beta-\alpha}=V(1)~\text{for}~\langle -\beta , \alpha\rangle=1;\] \[\mathbb{C}_{-\beta}=V(0)\otimes \mathbb{C}_{\omega_{i_{m}}}~\text{for}~\langle -\beta , \alpha\rangle=1;\]
\[\mathbb{C}_{-\beta}=V(0)\otimes \mathbb{C}_{-\omega_{i_{m}}}~\text{for}~\langle -\beta , \alpha\rangle=-1;\] \[\mathbb{C}_{-\beta}=V(0)~\text{for}~\langle -\beta , \alpha\rangle=0;\]
where $V(i)$ denotes an $i+1$-dimensional irreducible $\widetilde{L}_{\alpha_{i_{m}}}$-module.
Therefore by using SES and Lemma \ref{lemma 2.3} we conclude that the weights of the $B$-module $H^0(s_{i_{m}}\cdots s_{i_{r}},\alpha)$ are among the negative roots other than negatives of simple roots with multiplicity one. Further by proceeding recursively using SES, Lemma \ref{lemma 2.4}, and  Lemma \ref{lemma 2.3} we conclude that the weights of the $B$-module $H^0(w,\alpha)$ are among the negative roots other than negatives of simple roots with multiplicity one.
\end{proof}

\begin{lemma}\label{lemma 4.11}
Assume that the rank of $G$ is at least two. Let $w=s_{i_1}s_{i_2}\cdots s_{i_{r}}$ be a reduced expression and $\alpha_{i_{r}}=\alpha.$ Then we have $H^j(G/B, H^0(w,\alpha))=0$ for all $j\ge 0$.
\end{lemma}
\begin{proof}
The weights of the $B$-module $H^0(w,\alpha)$ are either of the three mutually exclusive cases as in Lemma \ref{lemma 4.10}. We prove the lemma by case by case.
\vspace{.2cm}
	
Case 1: Assume that the weights of the $B$-module $H^0(w,\alpha)$ are $\alpha,$ $0,$ and $-\alpha.$
	
Since the rank of $G$ is at least two, there is a $\beta$ in $S$ such that $\langle \alpha, \beta \rangle=-1.$ Consider the natural projection map \[p: G/B \longrightarrow G/P_{\{\alpha,\beta\}}.\]  Let $v=s_{\alpha}s_{\beta}s_{\alpha}.$  Then note that $P_{\{\alpha, \beta\}}/B=X(v).$
By using Lemma \ref{lemma 2.4} \[H^0(w,\alpha)=V(2),\] where $V(2)$ denotes a three dimensional irreducible $\widetilde{L}_{\alpha}$-module. Thus by Lemma \ref{lemma 2.3} we have  \[H^0(s_{\alpha}, H^0(w,\alpha))=\mathbb{C}_{\alpha}\oplus\mathbb{C}_{0}\oplus \mathbb{C}_{-\alpha};\] \[H^1(s_{\alpha}, H^0(w,\alpha))=0.\] 
Since $\langle \alpha, \beta\rangle=-1,$ by using Lemma \ref{lemma 2.4} the indecomposable $\widetilde{B}_{\beta}$-summands of $H^0(s_{\alpha},H^0(w,\alpha))$ are one of the following forms: \[\mathbb{C}_{\alpha}=V(0)\otimes\mathbb{C}_{-\omega_{\beta}};\] \[\mathbb{C}_{0}=V(0);\] \[\mathbb{C}_{-\alpha}=V(0)\otimes\mathbb{C}_{\omega_{\beta}}.\]  By using Lemma \ref{lemma 2.3} we have \[H^0(s_{\beta}, H^0(s_{\alpha}, H^0(w,\alpha)))=\mathbb{C}_{0}\oplus \mathbb{C}_{-\alpha}\oplus \mathbb{C}_{-\alpha-\beta}\]\[H^1(s_{\beta},H^0(s_{\alpha} ,H^0(w,\alpha)))=0.\] Therefore by using SES we have \[ H^0(s_{\beta}s_{\alpha}, H^0(w,\alpha))=\mathbb{C}_{0}\oplus \mathbb{C}_{-\alpha}\oplus \mathbb{C}_{-\alpha-\beta}\]\[ H^1(s_{\beta}s_{\alpha} ,H^0(w,\alpha))=0;\]  further by using SES recursively it follows that  \[H^j(s_{\beta}s_{\alpha} ,H^0(w,\alpha))=0~~(j\ge 2).\]
By using Lemma \ref{lemma 2.4} the indecomposable $\widetilde{B}_{\alpha}$-summands of $H^0(s_{\beta}s_{\alpha},H^0(w,\alpha))$ are the following
\[\mathbb{C}_{0}\oplus \mathbb{C}_{-\alpha}=V(1)\otimes \mathbb{C}_{-\omega_{\alpha}};\space \mathbb{C}_{-\alpha-\beta}=V(0)\otimes\mathbb{C}_{-\omega_{\alpha}}.\]
Therefore by using Lemma \ref{lemma 2.3} we have \[H^0(s_{\alpha}, H^0(s_{\beta}s_{\alpha},H^0(w,\alpha)))=0\] \[H^1(s_{\alpha}, H^0(s_{\beta}s_{\alpha},H^0(w,\alpha)))=0.\] By using SES we conclude that \[H^0(s_{\alpha}s_{\beta}s_{\alpha} ,H^0(w,\alpha))=0\] \[H^1(s_{\alpha}s_{\beta}s_{\alpha} ,H^0(w,\alpha))=0.\] Further, since  $H^j(s_{\beta}s_{\alpha} ,H^0(w,\alpha))=0$ for all $j\ge 2,$ by using SES we conclude that $H^j(s_{\alpha}s_{\beta}s_{\alpha} ,H^0(w,\alpha))=0$ for all $j\ge 2.$ In other words we have \[H^0(P_{\{\alpha,\beta\}}/B, H^0(w,\alpha))=0\]\[ H^j(P_{\{\alpha,\beta\}}/B, H^0(w,\alpha))=0 ~(j\ge 1).\]
Thus $R^{j}p_{*}\mathcal{L}(H^0(w,\alpha))=0$ for all $j\ge 1.$ Therefore by an application of a degenerate case of the Leray spectral sequence we have \[H^j(G/B, H^0(w,\alpha))=H^j(G/P_{\{\alpha,\beta\}}, H^0(P_{\{\alpha,\beta\}}/B,H^0(w,\alpha)))~\text{for ~all}~ j\ge 0.\]
Since $H^0(P_{\{\alpha,\beta\}}/B,H^0(w,\alpha))=0,$ it follows that $H^j(G/B, H^0(w,\alpha))=0$ for all $j\ge 0.$
\vspace{.2cm}
	
Case 2: Assume that the weights of the $B$-module $H^0(w,\alpha)$ are $0,$ $-\alpha,$ or among the negative roots other than negatives of simple roots with the multiplicity of each weight is one.
	
Consider the natural projection map \[p: G/B \longrightarrow G/P_{\alpha}.\] 
By using Lemma \ref{lemma 2.4} the indecomposable $\widetilde{B}_{\alpha}$-summands of $H^0(w,\alpha)$ are one of the following forms: \[\mathbb{C}_{0}\oplus \mathbb{C}_{-\alpha}=V(1)\otimes \mathbb{C}_{-\omega_{\alpha}};\]
\[\mathbb{C}_{-\beta}\oplus \mathbb{C}_{-\beta-\alpha}=V(1);\]
\[\mathbb{C}_{-\beta}=V(0)\otimes \mathbb{C}_{\omega_{\alpha}}~ \text{for }~ \langle -\beta, \alpha\rangle=1;\]
\[\mathbb{C}_{-\beta}=V(0)\otimes \mathbb{C}_{-\omega_{\alpha}} ~ \text{for }~ \langle -\beta, \alpha\rangle=-1;\]  \[\mathbb{C}_{-\beta}=V(0)~\text{for}~\langle -\beta , \alpha\rangle=0;\] where $V(i)$ denotes an $i+1$-dimensional irreducible $\widetilde{L}_{\alpha}$-module.
Therefore by Lemma \ref{lemma 2.3} we conclude that the weights of the $B$-module $H^0(s_{\alpha}, H^0(w,\alpha))$ are among the negative roots other than negatives of simple roots with multiplicity one and $H^1(s_{\alpha}, H^0(w,\alpha))=0.$  
Thus $R^{j}p_{*}\mathcal{L}(H^0(w,\alpha))=0$ for all $j\ge 1.$ Therefore by an application of a degenerate case of the Leray spectral sequence we have \[H^j(G/B, H^0(w,\alpha))=H^j(G/P_{\alpha}, H^0(P_{\alpha}/B,H^0(w,\alpha)))~\text{for ~all}~ j\ge 0.\]
Since $H^0(P_{\alpha}/B, H^0(w,\alpha))$ is a $B$-module such that the weights are among the negative roots other than negatives of simple roots, by using Fact 1, Fact 2, and Lemma \ref{lem 4.4} we have $H^j(G/P_{\alpha}, H^0(P_{\alpha}/B, H^0(w,\alpha)))=0$ for all $j\ge 0.$
\vspace{.2cm}
		
Case 3: Assume that the weights of the $B$-module $H^0(w,\alpha)$ are among the negative roots other than negatives of simple roots with multiplicity one.
Then by using Fact 1, Fact 2, and Lemma \ref{lem 4.4} we conclude that $H^j(G/B, H^0(w,\alpha))= 0 $ for all $j\ge 0.$
\end{proof}

Let $w=s_{i_{1}}s_{i_{2}}\cdots s_{i_{r-1}}s_{i_{r}}$ be a reduced expression and $\underline{i}=(i_{1},i_{2}\ldots,i_{r-1},i_{r})$.
\begin{lemma}\label{lemma 4.8}
Assume that the rank of $G$ is at least two. Then we have \[H^j(G/B, H^0(Z(w,\underline{i}), \Theta_{Z(w,\underline{i})}))=0 ~\text{for}~ j\ge0.\]
\end{lemma}
\begin{proof}
By Lemma \ref{lemma 4.6}, we have the following short exact sequence
\[0\to H^0(w,\alpha)\to H^0(Z(w,\underline{i}), \Theta_{Z(w,\underline{i})})\to H^0(Z(v,\underline{i}'), \Theta_{Z(v,\underline{i}')})\to 0 \]
of $B$-modules, where $\alpha=\alpha_{i_{r}},$ $v=s_{i_{1}}s_{i_{2}}\cdots s_{i_{r-1}},$ and $\underline{i}'=(i_{1},i_{2}\ldots,i_{r-1})$.

Therefore we have the following long exact sequence \[0\to H^0(G/B,H^0(w,\alpha))\to H^0(G/B,H^0(Z(w,\underline{i}), \Theta_{Z(w,\underline{i})}))\to H^0(G/B, H^0(Z(v,\underline{i}'), \Theta_{Z(v,\underline{i}')}))\]
\[\to H^1(G/B,H^0(w,\alpha))\to H^1(G/B, H^0(Z(w,\underline{i}), \Theta_{Z(w,\underline{i})}))\to \cdots \] of $G$-modules. By using Lemma \ref{lemma 4.11} we have \[H^j(G/B,H^0(Z(w,\underline{i}), \Theta_{Z(w,\underline{i})}))\simeq H^j(G/B, H^0(Z(v,\underline{i}'), \Theta_{Z(v,\underline{i}')})) ~\text{ for all }~j\ge 0.\] Hence by using induction on the length of the sequence we have $H^j(G/B,H^0(Z(w,\underline{i}), \Theta_{Z(w,\underline{i})}))\simeq H^j(G/B, H^0(s_{i_{1}},\alpha_{i_{1}}))$ for all $j\ge 0$. Therefore  by using Lemma \ref{lemma 4.11} we conclude the proof.
\end{proof}
	
	\begin{proposition}\label{prop 4.9}
		Assume that the rank of $G$ is at least two. Then we have the following:
	\begin{itemize}
			\item [(i)] ${\rm Aut}^0(\mathcal{Z}(w,\underline{i}))= G.$
			
			\item [(ii)] $H^j(\mathcal{Z}(w, \underline{i}), \Theta_{\mathcal{Z}(w,\underline{i})})= 0$ for $j\ge 1.$
	\end{itemize} 
	\end{proposition}
	\begin{proof}
		Proof of (i): By Lemma \ref{lemma 4.8}, we have $H^0(G/B, H^0(Z(w, \underline{i}), \Theta_{Z(w,\underline{i})})) = 0.$ Therefore, by Theorem \ref{thm 3.3}(i), we have ${\rm Aut}^0(\mathcal{Z}(w,\underline{i}))= G.$
		
		Proof of (ii): We argue as in the proof of Theorem \ref{thm 3.3}(ii).
		
		Recall that $\mathcal{Z}(w,\underline{i})=G\times^{B} Z(w,\underline{i}).$ Consider the natural first component projection \[\pi:\mathcal{Z}(w,\underline{i}) \to
		G/B.\] Therefore we have the following exact sequence \[0\rightarrow \mathcal{R}\rightarrow \Theta_{\mathcal{Z}(w,\underline{i})} \rightarrow {\pi}^{*}\Theta_{G/B} \rightarrow 0\]
		of vector bundles on $\mathcal{Z}(w,\underline{i}),$ where $\mathcal{R}$ denotes the relative tangent bundle on $\mathcal{Z}(w,\underline{i})$ with respect to $\pi.$ Since $H^j(G/B, \Theta_{G/B})$ and $H^{j}(Z(w,\underline{i}), \mathcal{O}_{Z(w,\underline{i})})$ vanish for $j\ge 1$ (see \cite{Dem aut},\cite[Theorem 2, p.75 and Theorem 1, p.130]{Akh}), by the Leray spectral sequence for $\pi$ and the projection formula, $H^j(\mathcal{Z}(w,\underline{i}),\pi^{*}\Theta_{G/B})$ also vanishes for $j\ge 1$. Moreover by using Proposition \ref{prop 3.1} and the long exact sequence associated to the above short exact sequence we have the following: 
		\[0\to H^0(\mathcal{Z}(w,\underline{i}),\mathcal{R})\to H^0(\mathcal{Z}(w,\underline{i}),\Theta_{\mathcal{Z}(w,i)})\to H^0(\mathcal{Z}(w,\underline{i}), \pi^{*}\Theta_{G/B})=\mathfrak{g}\to 0\]
		\[ H^j(\mathcal{Z}(w,\underline{i}),\mathcal{R})\simeq H^j(\mathcal{Z}(w,\underline{i}),\Theta_{\mathcal{Z}(w,i)})~ \text{for all}~ j\ge 1.\]

		Since $\mathcal{R}$ is the relative tangent bundle on $\mathcal{Z}(w,\underline{i}),$ the restriction of $\mathcal{R}$ to $Z(w,\underline{i})$ is $\Theta_{Z(w,\underline{i})}$. By Lemma \ref{lemma 4.6}, we have \[H^j(Z(w, \underline{i}), \mathcal{R}|_{Z(w,\underline{i})}) = H^j(Z(w, \underline{i}), \Theta_{Z(w,\underline{i})}) = 0 ~\text{ for}~ j \ge 1.\] Hence we have \[\pi_{*}\mathcal{R}=\mathcal{L}(H^0(Z(w,\underline{i}),\Theta_{Z(w,\underline{i})}))\]\[R^j\pi_{*}\mathcal{R}=0 ~\text{for}~ j\ge 1.\] Therefore, by using a degenerate case of the Leray spectral sequence we obtain \[H^j(\mathcal{Z}(w,\underline{i}),\mathcal{R})=H^j(G/B, H^0(Z(w,\underline{i}), \Theta_{Z(w,\underline{i})})).\] Hence by using Lemma \ref{lemma 4.8}, we conclude the proof.
	\end{proof}
	
	\begin{corollary}
		Assume that the rank of $G$ is at least two and $w\neq id.$ Then $\mathcal{Z}(w, \underline{i})$ is not isomorphic to $G/B \times Z(w, \underline{i}).$
	\end{corollary}
	\begin{proof}
		If $\mathcal{Z}(w, \underline{i})=G/B \times Z(w, \underline{i}),$ then by \cite[Corollary 2.3, p.46]{Bri} we have ${\rm Aut}^0(\mathcal{Z}(w, \underline{
			i}))={\rm Aut}^0(G/B) \times {\rm Aut}^0(Z(w, \underline{i})).$ Since ${\rm Aut}^0(G/B)=G$ (see \cite{Dem aut}), we have ${\rm Aut}^0(\mathcal{Z}(w,\underline{i}))=G \times {\rm Aut}^0(Z(w, \underline{
			i})).$ Further, since $w\neq id,$ $Z(w,\underline{i})$ contains an open $B$-orbit of positive dimension, whence $B$ acts non-trivially on $Z(w,\underline{i}).$ Therefore, we conclude that ${\rm Aut}^0(Z(w,\underline{i}))$ is not a trivial group. Hence, ${\rm Aut}^0(\mathcal{Z}(w, \underline{i}))\neq G,$ which shows contradiction to Proposition \ref{prop 4.9}(i).
	\end{proof}

\begin{remark}
If the rank of $G$ is one, then for both $w=id$ or $w=s_{\alpha},$  $\mathcal{Z}(w, \underline{i})$ is isomorphic to $G/B \times Z(w, \underline{i}).$ Moreover, for $w=s_{\alpha}$, we have $\mathcal{Z}(w,\underline{i})\simeq G/B\times G/B\simeq \mathbb{P}^{1}\times \mathbb{P}^{1}$. Hence \[{\rm Aut}^0(\mathcal{Z}(w,\underline{i}))= G \times G ~\text{and}~H^j(\mathcal{Z}(w, \underline{i}), \Theta_{\mathcal{Z}(w,\underline{i})})= 0~\text{ for }~ j\ge 1.\] Further for $w=id$, we have $\mathcal{Z}(w,\underline{i})\simeq G/B\simeq \mathbb{P}^{1}$.
Therefore \[{\rm Aut}^0(\mathcal{Z}(w,\underline{i}))= G ~\text{ and} ~H^j(\mathcal{Z}(w, \underline{i}), \Theta_{\mathcal{Z}(w,\underline{i})})= 0~\text{ for }~ j\ge 1.\]
\end{remark}

	We conclude this article by giving an example which shows that if $G$ is not simply-laced, then $H^1(\mathcal{Z}(w,\underline{i}), \Theta_{\mathcal{Z}(w,\underline{i})})$ might not vanish, i.e., $\mathcal{Z}(w,\underline{i})$ is not locally rigid.
	
	\begin{example}\label{ex4.12}
		Let $G={\rm SO}(5,\mathbb{C}).$ Let $w= s_1s_2s_1,$ $\underline{i}=(1, 2, 1).$ Then $H^1(\mathcal{Z}(w, \underline{i}), \Theta_{\mathcal{Z}(w,\underline{i})})\neq 0.$
	\end{example} 
	\begin{proof}
	Consider the following $B$-equivariant fibre product diagram:	
		\[
		\begin{tikzcd}
		Z(w,\underline{i}) \arrow[r,"\phi_{w}"] \arrow[d,"f_{3}"'] &
		G/B  \arrow[d,"\pi"'] &
		\\
		Z(w_{1},\underline{i}_{1}) \arrow[r, "\pi\circ\phi_{w_1}"] &
		G/P_{\alpha_{1}}&
		\end{tikzcd}
		\]
where $w_1= s_1s_2,$  and $\underline{i}_{1}=(1, 2)$. Since $\pi$ is $P_{\alpha_{1}}/B~(\simeq\mathbb{P}^{1})$-fibration, the relative tangent bundle on $G/B$ with respect to $\pi$ is the line bundle $\mathcal{L}(\alpha_{1}).$ Therefore from the above diagram we have the following exact sequence \[0\rightarrow \mathcal{L}(w,\alpha_1)\rightarrow \Theta_{Z(w,\underline{i})} \rightarrow {f_{3}}^{*}\Theta_{Z(w_1,\underline{i}_{1})} \rightarrow 0\]
of vector bundles on $Z(w,\underline{i}).$
Note that by \cite[Corollary 6.4, p.780]{Kan} it follows that $H^2(w,\alpha_{1})=0$.
This gives rise to a long exact sequence\[0\rightarrow H^0(w,\alpha_1)\rightarrow H^0(Z(w,\underline{i}),\Theta_{Z(w,\underline{i})})\rightarrow H^0(Z(w_1,\underline{i}_1),\Theta_{Z(w_1,\underline{i}_1)})\] \[\rightarrow H^1(w,\alpha_1)\rightarrow H^1(Z(w,\underline{i}),\Theta_{Z(w,\underline{i})})\rightarrow H^1(Z(w_1,\underline{i}_1),\Theta_{Z(w_1,\underline{i}_1)})\rightarrow 0,\]
of $B$-modules.	

Note that by using Lemma \ref{lemma 2.1} or by Borel-Weil-Bott we have \[H^0(s_1,\alpha_{1})=\mathbb{C}_{\alpha_{1}}\oplus \mathbb{C}_{0}\oplus \mathbb{C}_{-\alpha_{1}} ~\text{and}~ H^1(s_{1},\alpha_{1})=0.\]

Note that the unipotent radical $R_{u}(P_{\alpha_{1}})$ of $P_{\alpha_{1}}$ acts trivially on the $B$-module $H^0(s_1,\alpha_{1})$ because it acts trivially on $P_{\alpha_{i_1}}/B.$ As an $\widetilde{B}_{\alpha_{2}}$-module, indecomposable summands of $H^0(s_{1},\alpha_1)$ are $\mathbb{C}_{\alpha_{1}}, \mathbb{C}_{-\alpha_{1}}$ and $\mathbb{C}_{0}.$ Therefore, by Lemma \ref{lemma 2.4} we have the following: \[\mathbb{C}_{\alpha_{1}}=V(0)\otimes \mathbb{C}_{\alpha_{1}}, \mathbb{C}_{0}=V(0)\otimes \mathbb{C}_{0}, ~\text{ and}~\mathbb{C}_{-\alpha_{1}}=V(0)\otimes\mathbb{C}_{-\alpha_{1}}\] where $V(0)$ denotes the trivial one-dimensional $\widetilde{L}_{\alpha_{2}}$-module. Since $\langle \alpha_{1}, \alpha_{2}\rangle=-2,$  by using Lemma \ref{lemma 2.3}  we have the following:
\[H^0(s_{2},H^0(s_{1},\alpha_{1}))=\mathbb{C}_{0}\oplus \mathbb{C}_{-\alpha_{1}}\oplus\mathbb{C}_{-\alpha_{1}-\alpha_{2}}\oplus \mathbb{C}_{-\alpha_{1}-2\alpha_{2}}~\text{ and}~ H^1(s_{2},H^0(s_{1},\alpha_{1}))=\mathbb{C}_{\alpha_{1}+\alpha_{2}} \] as $s_{2}\cdot \alpha_{1}=\alpha_{1}+\alpha_{2}.$

By using (SES) for $w=s_{2}s_{1},$ $B$-module $V=\mathbb{C}_{\alpha_{1}}$ and $j=0,1,$ we have \[ H^0(s_{2}s_{1},\alpha_{1})=H^0(s_{2},H^0(s_{1},\alpha_{1}))~\text{and} ~H^1(s_{2}s_{1},\alpha_{1})=H^1(s_{2}, H^0(s_{1},\alpha_{1})).\] 

Therefore we have 
\[ H^0(s_{2}s_{1},\alpha_{1})= \mathbb{C}_{0}\oplus \mathbb{C}_{-\alpha_{1}}\oplus\mathbb{C}_{-\alpha_{1}-\alpha_{2}}\oplus \mathbb{C}_{-\alpha_{1}-2\alpha_{2}} ~\text{and} ~H^1(s_{2}s_{1},\alpha_{1})=\mathbb{C}_{\alpha_{1}+\alpha_{2}}.\] 			
Note that as an $\widetilde{B}_{\alpha_{1}}$-module the indecomposable summands of  $H^0(s_{2}s_{1},\alpha_{1})$ are the following \[\mathbb{C}_{0}\oplus \mathbb{C}_{-\alpha_{1}}, \mathbb{C}_{-\alpha_{1}-\alpha_{2}},~\text{and}~ \mathbb{C}_{-\alpha_{1}-2\alpha_{2}}.\]
Moreover by Lemma \ref{lemma 2.4}
\[\mathbb{C}_{0}\oplus \mathbb{C}_{-\alpha_{1}}=V(1)\otimes \mathbb{C}_{-\omega_{1}}, \mathbb{C}_{-\alpha_{1}-\alpha_{2}}=V(0)\otimes \mathbb{C}_{-\omega_1}, ~\text{ and}~\mathbb{C}_{-\alpha_{1}-2\alpha_{2}}=V(0)\otimes\mathbb{C}_{-\alpha_{1-2\alpha_{2}}}\] where $V(i)$ denotes an $i+1$-dimensional irreducible $\widetilde{L}_{\alpha_{1}}$-module. Therefore by using Lemma \ref{lemma 2.3}, we have the following \[H^0(s_1, H^{0}(s_2s_1,\alpha_{1}))=\mathbb{C}_{-\alpha_{1}-2\alpha_{2}}~ \text{and}~ H^1(s_1, H^0(s_2s_1,\alpha_{1}))=0\]

Similarly, since $\langle \alpha_{1}+\alpha_{2},\alpha_{1}\rangle=-1,$ by using Lemma \ref{lemma 2.3} we have \[H^0(s_1, H^1(s_2s_1,\alpha_{1}))=\mathbb{C}_{\alpha_{1}+\alpha_{2}}\oplus \mathbb{C}_{\alpha_{2}}.\]

Therefore by using (SES) we have the following \[H^0(w,\alpha_{1})=\mathbb{C}_{-\alpha_{1}-2\alpha_{2}} ~\text{and}~ H^1(w, \alpha_{1})=\mathbb{C}_{\alpha_{1}+\alpha_{2}}\oplus \mathbb{C}_{\alpha_{2}}.\]

Recall the following short exact sequence \[0\rightarrow \mathcal{L}(w_1,\alpha_2)\rightarrow \Theta_{Z(w_1,\underline{i}_{1})} \rightarrow {f_{2}}^{*}\Theta_{Z(w_2,\underline{i}_{2})} \rightarrow 0\]
of vector bundles on $Z(w_1,\underline{i}_{1}),$ where $w_2=s_1$ and $\underline{i}_{2}=(1).$

Note that by \cite[Corollary 6.4, p.780]{Kan} it follows that $H^2(w_1,\alpha_{2})=0$.
This gives rise to a long exact sequence\[0\rightarrow H^0(w_1,\alpha_2)\rightarrow H^0(Z(w_1,\underline{i}_1),\Theta_{Z(w_1,\underline{i}_1)})\rightarrow H^0(Z(w_2,\underline{i}_2),\Theta_{Z(w_2,\underline{i}_2)})\] \[\rightarrow H^1(w_1,\alpha_2)\rightarrow H^1(Z(w_1,\underline{i}_1),\Theta_{Z(w_1,\underline{i}_1)})\rightarrow H^1(Z(w_2,\underline{i}_2),\Theta_{Z(w_2,\underline{i}_2)})\rightarrow 0,\]
of $B$-modules.	

Note that $Z(w_2,\underline{i}_{2})\simeq P_{\alpha_{1}}/B.$ Therefore we have \[H^0(Z(w_2,\underline{i}_2),\Theta_{Z(w_2,\underline{i}_2)})=H^0(s_{1},\alpha_1)=\mathbb{C}_{\alpha_{1}}\oplus \mathbb{C}_{0}\oplus \mathbb{C}_{-\alpha_{1}}\] and \[H^1(Z(w_2,\underline{i}_2),\Theta_{Z(w_2,\underline{i}_2)})=H^1(s_{1},\alpha_1)=0.\] Since $\alpha_{2}$ is short root, by \cite[Corollary 5.6, p. 778]{Kan} it follows that $H^1(w_1,\alpha_{2})=0$.
Therefore from the above discussion we have the following 
\[H^1(Z(w_1, \underline{i}_{1}), \Theta_{Z(w_1,\underline{i}_{1})}) = 0.\]

By using (SES) we have $H^0(w_{1},\alpha_{2})=\mathbb{C}_{0}\oplus\mathbb{C}_{-\alpha_{2}}\oplus\mathbb{C}_{-\alpha_{1}-\alpha_{2}}$. Therefore from the above long exact sequence corresponding to $Z(w_1,\underline{i}_1)$, we have
\[H^0(Z(w_1, \underline{i}_{1}), \Theta_{Z(w_1,\underline{i}_{1})})_{\mu} = 0\] for $\mu = \alpha_1+\alpha_2,\alpha_2.$ Thus from the above long exact sequence corresponding to $Z(w,\underline{i})$, we have the following short exact sequence 
\[	0\to H^1(w, \alpha_{1})\to H^1(Z(w, \underline{i}), \Theta_{Z(w,\underline{i})})\to H^1(Z(w_1,\underline{i}_{1}), \Theta_{Z(w_1,\underline{i}_1)})\to 0\] of $B$-modules. Since $H^1(Z(w_1, \underline{i}_{1}), \Theta_{Z(w_1,\underline{i}_{1})})= 0,$ we have \[H^1(Z(w,\underline{i}), \Theta_{Z(w,\underline{i})})= H^1(w,\alpha_{1}) = \mathbb{C}_{\alpha_{1}+\alpha_{2}}\oplus \mathbb{C}_{\alpha_{2}}.\]
		
Recall that $\mathcal{Z}(w,\underline{i})=G\times^{B} Z(w,\underline{i}).$ Consider the natural first component projection \[\pi:\mathcal{Z}(w,\underline{i}) \to
G/B.\] Therefore we have the following exact sequence \[0\rightarrow \mathcal{R}\rightarrow \Theta_{\mathcal{Z}(w,\underline{i})} \rightarrow {\pi}^{*}\Theta_{G/B} \rightarrow 0\]
of vector bundles on $\mathcal{Z}(w,\underline{i}),$ where $\mathcal{R}$ denotes the relative tangent bundle on $\mathcal{Z}(w,\underline{i})$ with respect to $\pi.$ Note that $H^1(\mathcal{Z}(w,\underline{i}),\pi^{*}\Theta_{G/B})$ vanishes as $H^1(G/B, \Theta_{G/B})$ and $H^1(Z(w,\underline{i}),\mathcal{O}_{Z(w,\underline{i})})$ vanish. Therefore, we have the following exact sequence
\[0\to H^0(\mathcal{Z}(w,\underline{i}),\mathcal{R})\to H^0(\mathcal{Z}(w,\underline{i}),\Theta_{\mathcal{Z}(w,i)})\to H^0(\mathcal{Z}(w,\underline{i}), \pi^{*}\Theta_{G/B})=\mathfrak{g}\]
\[\to H^1(\mathcal{Z}(w,\underline{i}),\mathcal{R})\to H^1(\mathcal{Z}(w,\underline{i}),\Theta_{\mathcal{Z}(w,i)})\to 0\] of $G$-modules.
Hence by using note (2) (see after Corollary \ref{cor 3.4}) we have $H^1(\mathcal{Z}(w,\underline{i}),\Theta_{\mathcal{Z}(w,\underline{i})})=H^1(\mathcal{Z}(w, \underline{i}),\mathcal{R}).$ By the above computation it follows that the non-zero weights of $H^0(Z(w,\underline{i}),\Theta_{Z(w,\underline{i})})$ are roots. Since the index of a non-singular root is at most one, by using Lemma \ref{lem 4.4} we have $H^2(G/B,H^0(Z(w,\underline{i}),\Theta_{Z(w,\underline{i})}))=0$. 

Therefore to show $H^1(\mathcal{Z}(w,\underline{i}),\Theta_{\mathcal{Z}(w,\underline{i})})$ does not vanish by five term exact sequence it is sufficient to show $H^0(G/B,H^1(Z(w,\underline{i}),\Theta_{Z(w,\underline{i})}))$ does not vanish. Note that $H^1(Z(w,\underline{i}), \Theta_{Z(w,\underline{i})})=\mathbb{C}_{\alpha_{1}+\alpha_{2}}\oplus \mathbb{C}_{\alpha_{2}}$.
 In order to compute the cohomology module $H^0(G/B, \mathbb{C}_{\alpha_{1}+\alpha_{2}}\oplus \mathbb{C}_{\alpha_{2}}),$ we fix a reduced expression $w_0= s_1s_2s_1s_2.$ 
 Then by using (SES) we have  \[H^0(G/B, \mathbb{C}_{\alpha_{1}+\alpha_2}\oplus \mathbb{C}_{\alpha_2})= \mathbb{C}_{\alpha_{1}+\alpha_{2}}\oplus \mathbb{C}_{\alpha_{2}}\oplus \mathbb{C}_{0}\oplus \mathbb{C}_{-\alpha_{2}}\oplus \mathbb{C}_{-(\alpha_{1}+\alpha_{2})}=V(\omega_1),\] where $V(\omega_1)$ denotes the finite dimensional irreducible $G$-module with highest weight $\omega_1.$
\end{proof}


\begin{thebibliography}{KP}
		
		
\bibitem{Akh} D. N. Akhiezer, {\sl Lie Group Actions in Complex Analysis}, Aspects of Mathematics {\bf E 27} Vieweg, Braunschweig/Wiesbaden, 1995.
		
		
\bibitem{BKS} V. Balaji, S. Senthamarai Kannan, K.V. Subrahmanyam, {\it Cohomology of line bundles on Schubert varieties-I}, Transformation Groups {\bf 9} (2004), no.2,  105-131.
		

		
\bibitem{Bot} R. Bott, {\it Homogeneous vector bundles.} Ann. Math.{\bf 66}(2) (1957), 203-248.
		
		\bibitem{BS} R. Bott, H. Samelson, {\it Applications of the theory of Morse to symmetric spaces.} Amer. J. Math. {\bf 80} (1958), 964-1029.
		
		\bibitem{Bri} M. Brion, {\it On automorphism groups of fiber bundles,} Publ.Math.Urug.{\bf 12} (2011), 39-66.
		
		\bibitem{BK} M. Brion, S. Kumar, {\sl Frobenius Splitting Methods in Geometry and Representation theory,} Progress in Mathematics, Vol. {\bf 231,} Birkh\"auser Boston, Inc., Boston, MA, 2005.
		
		\bibitem{CKP} B. N. Chary, S. S. Kannan, A. J. Parameswaran, {\it Automorphism group of a Bott-Samelson-Demazure-Hansen variety,} Transformation Groups {\bf 20} (2015), no. 3, 665-698.
		
	    \bibitem{CPS}	E. Cline, B. Parshall, L. Scott, {\it Induced Modules and Extensions of Representations, II} J. London Math. Soc. (2), {\bf 20} (1979), 403-414.

		
		\bibitem{Dem1} M. Demazure, {\it D\'esingularisation des vari\'et\'es de Schubert g\'en\'eralis\'ees.}
		Ann. Sci. \'Ecole Norm. Sup. ({\bf 4}) 7 (1974), 53-88.
		
		
		
		\bibitem{Dem2} M. Demazure, {\it A very simple proof of Bott's theorem}, Invent. Math. {\bf 33} (1976), 271-272.
		
		
		\bibitem{Dem aut} M. Demazure, {\it Automorphismes et d\'eformations des vari\'et\'es de Borel}, Invent. Math. 39:2 (1977), 179-186.
		
		\bibitem{Gri} P. Griffiths, {\it Some geometric and analytic properties of homogeneous complex manifolds}, {\bf I}, Acta
		Math. {\bf 110} (1963), 115-155.
		
		\bibitem{Han} H. C. Hansen, {\it On cycles on flag manifolds,} Math. Scand. {\bf 33} (1973),269-274.
		
		\bibitem{Har} R. Hartshorne, {\sl Algebraic Geometry,} Graduate Texts in Mathematics book series (GTM, vol. {\bf 52}) (1977).
		
		
		\bibitem{Hum1} J. E. Humphreys, {\sl Introduction to Lie algebras and Representation theory}, Springer-Verlag, Berlin Heidelberg, New York, 1972.
		
		
		
		\bibitem{Hum2} J. E. Humphreys, {\sl Linear Algebraic Groups}, Springer-Verlag, Berlin Heidelberg,  New York, 1975.
		
		
		\bibitem{Huy} D. Huybrechts, {\sl Complex Geometry: An Introduction,} Springer-Verlag, Berlin Heidelberg, New York, 2005.
		
		
		
		\bibitem{Jan} J. C. Jantzen, {\sl Representations of Algebraic Groups}, (Second Edition), Mathematical Surveys and Monographs, Vol.{\bf 107}, 2003.
		
		
		
		
		
		\bibitem{Kan}S. S. Kannan, {\it On the automorphism group of a smooth Schubert Variety}, Algebr. Represent. Theory {\bf19} (2016), no 4, 761-782.
		
	
		
		
		
		\bibitem{MO} H. Matsumura, F. Oort, {\it Representability of group functors, and automorphisms of algebraic schemes}, Invent. Math. {\bf 4} (1967), 1-25. \label{Mat}
		
		
		 \bibitem{Ser} J. P. Serre, {\it On the fundamental group of a unirational variety}, J. London Math. Soc. {\bf 34} (1959), 481-484.
		 
		 
		\bibitem{Snow} D. M. Snow, {\sl Homogeneous vector bundles,} \url{https://www3.nd.edu/~snow/Papers/HomogVB.pdf}.
		

		
\end{thebibliography}
\end{document}